\documentclass[reqno, 10pt]{amsart}
\usepackage[top=3.75cm, bottom=3.75cm, left=2.5cm, right=2.5cm, headsep=0.75cm]{geometry}            
\usepackage{graphicx}
\graphicspath{{Images/}}
\usepackage{amsmath,amssymb,amsthm,amsfonts,mathrsfs,mathtools,esint}
\numberwithin{equation}{section}
\usepackage{parskip}
\usepackage{comment}
\usepackage{setspace}
\usepackage{etoolbox}
\appto\normalsize{\belowdisplayshortskip=\belowdisplayskip}
\usepackage{url}
\usepackage{hyperref}
\hypersetup{
    colorlinks=true,
    linkcolor=blue,
    filecolor=magenta,      
    urlcolor=cyan,
    pdftitle={Overleaf Example},
    pdfpagemode=FullScreen,
    }

\usepackage{fancyhdr}
\usepackage{enumitem}

\let\tmp\phi \let\phi\varphi \let\varphi\tmp
\let\tmp\epsilon \let\epsilon\varepsilon \let\varepsilon\tmp
\newcommand{\fra}{(-\Delta)^s}
\newcommand{\R}{\mathbb{R}}
\newcommand{\RR}{\mathbb{R}^N}
\renewcommand{\H}{H^s(\RR)}

\newcommand{\HHH}{\dot{H}^s(\RR)}

\newcommand{\NN}{\mathbb{N}}

\newcommand{\MM}{\mathcal{M}}
\renewcommand{\NN}{\mathcal{N}}

\renewcommand{\SS}{\mathcal{S}}

\theoremstyle{definition}
\newtheorem{definition}{Definition}[section]

\newtheorem{lemma}{Lemma}[section]
\newtheorem{theorem}{Theorem}[section]
\newtheorem{corollary}{Corollary}[theorem]
\newtheorem{remark}{Remark}[section]

\usepackage[numbers]{natbib}

\begin{document}
\title[Fractional Schr\"odinger equations]{Fractional Schr\"odinger equations with mixed nonlinearities: asymptotic profiles, uniqueness and nondegeneracy of ground states}

\author{Mousomi Bhakta, Paramananda Das and Debdip Ganguly}
\address{M. Bhakta, Department of Mathematics, Indian Institute of Science Education and Research, Dr. Homi Bhaba Road, Pune-411008, India}
\email{mousomi@iiserpune.ac.in}
\address{P. Das, Department of Mathematics, Indian Institute of Science Education and Research, Dr. Homi Bhaba Road, Pune-411008, India}
\email{paramananda.das@students.iiserpune.ac.in,  pd348225@gmail.com}
\address{D. Ganguly, Department of Mathematics\\
Indian Institute of Technology Delhi\\
IIT Campus, Hauz Khas, New Delhi, Delhi 110016, India}
\email{debdip@maths.iitd.ac.in}

\subjclass[2010]{Primary 35J60, 35B08, 35B20, 35B40, 35B44, 35J10, 35J20}
\keywords{critical, subcritical nonlinearity, blow-up, uniqueness, nondegenracy, asymptotic profile, rate of convergence, fractional Schr\"odinger.}
\date{}

\begin{abstract}
We study the fractional Schr\"odinger equations with a vanishing parameter:
$$
\left\{\begin{aligned}
       \fra u+u &=|u|^{p-2}u+\lambda|u|^{q-2}u \text{ in }\RR \\
     u &\in H^s(\mathbb{R}^N),
                \end{aligned}
  \right.
\leqno{(\mathcal{P}_\lambda)}
$$
where $s\in(0,1)$, $N>2s$, $2<q<p\leq 2^*_s=\frac{2N}{N-2s}$ are fixed parameters and  $\lambda>0$ is a vanishing parameter. We investigate the asymptotic behaviour of positive ground state solutions for $\lambda$ small, when $p$ is subcritical, or critical Sobolev exponent $2^*_s$. For $p<2_s^*$, the ground state solution asymptotically coincides with unique positive ground state solution of $(-\Delta)^s u+u=u^p$, whereas for $p=2_s^*$ the asymptotic behaviour of the solutions, after a rescaling, is given by the unique positive solution of the nonlocal critical Emden-Fowler type equation.  Additionally, for $\lambda>0$ small, we show the uniqueness and nondegeneracy of the positive ground state solution using these asymptotic profiles of solutions.
\end{abstract}

\maketitle
\setcounter{tocdepth}{1}
\tableofcontents

\section{Introduction}
In this article we consider the following fractional Schr\"odinger equations with mixed nonlinearities 
\begin{equation}\tag{$P_\lambda$}\label{PDE}
    \fra u+u=|u|^{p-2}u+\lambda|u|^{q-2}u \quad  \text{ in }\RR,
\end{equation}
where $N>2s$, $2<q<p\leq2_s^\ast=\frac{2N}{N-2s}$ and $\lambda>0$ is a parameter. Here $(-\Delta)^s$ denotes the fractional Laplace operator which can be defined for Schwartz class functions as follows
$$\left(-\Delta\right)^su(x): = a_{N,s}\mbox{P.V}\int_{\mathbb{R}^N}\frac{u(x)-u(y)}{|x-y|^{N+2s}} \; dy,$$
where $P.V$ stands for principal value and $a_{N,s}= \frac{4^s\Gamma(N/2+ s)}{\pi^{N/2}|\Gamma(-s)|}$.

Fractional Klein-Gordon equations and more general versions of fractional Schr\"odinger equations are related to Equation~\eqref{PDE}, namely (Schr\"odinger equations) 
$$
i\frac{\partial\psi}{\partial t}+(-\Delta)^s\psi+(V(x)+\omega)\psi=f(x,\psi),
$$
where $\psi=\psi(x,t)$ is a complex valued function defined on $\mathbb{R}^N\times\mathbb{R}$.  Suppose we assume 
\begin{equation}\label{9-16-1}
f(x, \rho e^{i\theta})=e^{i\theta}f(x,\rho), \quad\forall\, \rho,\, \theta\in\mathbb{R}, \quad x\in\mathbb{R}^N,\end{equation} 
and  $f: \mathbb{R}^N\times\mathbb{R}\to\mathbb{R}$ and $f(x,.)$ is a continuous odd function and $f(x,0)=0.$ Then one can look 
for standing wave solutions, i.e., $\psi(x,t)=e^{i\omega t}u(x)$, which led us to the following scalar field equation 
\begin{equation}\label{9-16-2}
(-\Delta)^s u + V(x)u = f(x,u)\;\;\text{in}\;\mathbb{R}^{N},
\end{equation}
see (\cite{LAS1, LAS2}).  Now consider the following fractional nonlinear Klein-Gordon equation
$$
\psi_{tt}+(-\Delta)^s\psi+(V(x)+\omega^2)\psi=f(x,\psi),
$$
where $\psi:\mathbb{R}^N\times\mathbb{R}\to \mathbb{C}$ and $f$ satisfies \eqref{9-16-1}.  After that, one can look for standing wave solutions as previously, which will again lead to the equation of type \eqref{9-16-2}.

Now, let's go over the current state of the art for the results for \eqref{PDE} in the local case, i.e.,  $s= 1.$  The existence of the ground state solution of \eqref{PDE} has been proven in \cite{ZZ} and in a recent paper \cite{WW} in the local case $s=1$ and $p=2^*=\frac{2N}{N-2}$.  Akahori et al. in \cite{AIIKN} have demonstrated that, for $s=1$ and $p=2^*$ up to a rescaling ground states of \eqref{PDE} 
converges to the unique positive solution of $-\Delta w=|w|^{2^*-2}w$ in $\RR$ as $\lambda\to 0$.  In a recent work, Moroz and Ma \cite{MM1} provided detailed asymptotic profiles of the ground states $u_\lambda$ in various norms and explored how these findings relate to a problem with a mass constraint. For $s=1$ and $p<2^*$, Fukuizumi \cite{F} have studied the asymptotic behavior of the ground states of \eqref{PDE}. Uniqueness of radial ground states for \eqref{PDE} in the case $s=1$ have been proved in \cite{AIIKN}. Besides, Moroz and Muratov \cite{MM2} studied the asymptotic properties of ground states of 
the equation in the form 
\begin{equation}\label{MM-local}
-\Delta u+\epsilon u=|u|^{p-2}u-|u|^{q-2}u \quad\mbox{in }\RR,
\end{equation}
where $\epsilon>0$ is a vanishing parameter and $q>p> 2$ are fixed parameters. They showed  that the
behavior of solutions depends on whether $p$ is less than, equal to or greater than the critical Sobolev exponent $2^*$.

In the nonlocal case $s\in (0,1)$, there has been a considerable interest in studying the existence and asymptotic behaviour of the ground states of fractional Schr\"odinger type equations. In a very recent work, Cassani and Wang \cite{CW1} have studied the behavior of the ground states and local uniqueness for fractional
Schr\"odinger equations with nearly critical growth, namely for the equation
$$(-\Delta)^s u+V(x)u=u^{2^*_s-1-\epsilon}\quad\mbox{in }\RR.$$

In \cite{BM2}, Bhakta and Mukherjee examined the nonlocal version of problem \eqref{MM-local}, where $q>p>2$, $N>2s$, and $-\Delta$ is replaced with $(-\Delta)^s$ in \eqref{MM-local}. They have examined the asymptotic behaviour of ground state solutions when $p$ is subcritical, supercritical, or critical Sobolev exponent. They have also proven the existence and several qualitative properties of positive solutions. Moreover, for the equation $$(-\Delta)^s u=u^{p-1}-\epsilon u^{q-1} \mbox{ in } \Omega, \quad u>0 \mbox{ in } \Omega, \quad u=0 \quad\text{in}\,\, \RR\setminus\Omega,$$ where $q>p>2^*_s$, Bhakta, et al. \cite{BMS} have examined the asymptotic profile of ground states. For more results in fractional Schr\"odinger type equations, we additionally refer to \cite{BCG, FLS, FQT, S} and the references therein. 

Before we start the discussion on \eqref{PDE}, we need to introduce some notations and definitions. Define, 
$$H^s(\mathbb{R}^N): =\bigg\{u\in L^{2}(\mathbb{R}^N) \; : [u]_{\H}^2 := \; \iint_{\mathbb{R}^{2N}}\frac{|u(x)-u(y)|^2}{|x-y|^{N+2s}}\,dx \, dy<\infty\bigg\},
$$
with the Gagliardo norm
$$
\|u\|_{H^{s}(\mathbb{R}{^N})} := \Big(\int_{\mathbb{R}{^N}}|u|^2 \, dx+\iint_{\mathbb{R}^{2N}} \frac{|u(x)-u(y)|^2}{|x-y|^{N+2s}}\,dx\,dy\Big)^{1/2}.$$
\begin{definition}
	We say $u \in H^s(\mathbb{R}^N)$ is a positive weak solution of \eqref{PDE} if $u>0$ in $\mathbb{R}^N$ and for every $\phi\in H^s(\mathbb{R}^N)$, we have
	\begin{eqnarray}
	 \iint_{\mathbb{R}^{2N}}\frac{(u(x) - u(y))(\phi(x)-\phi(y))}{|x-y|^{N+2s}}\,dx\,dy +\int_{\mathbb{R}^N} u\phi \, dx &=& \int_{\mathbb{R}^N}u^{p-1}\phi \, dx+\lambda\int_{\mathbb{R}^N}u^{q-1}\phi \, dx .\nonumber
	\end{eqnarray}
	\end{definition}
The energy functional associated to the above equation is 
\begin{equation}
        I_{\lambda}(u):=\frac12[u]_{\H}^2+\frac12\int_{\RR}|u|^2\,dx-\frac1{p}\int_{\RR}|u|^{p}\,dx-\frac{\lambda}{q}\int_{\RR}|u|^q\,dx.
\end{equation}
It is easy to see that $\{I_{\lambda}(v):v\in\H,v\not\equiv0,I_\lambda'(v)=0\}$ is nonempty. We say $u\in\H$ is a ground state solution of \eqref{PDE} if 
\begin{equation}\label{ground_state}
    I_{\lambda}(u)=m^\lambda:=\inf\{I_{\lambda}(v)\,:\, v\in\H,\, v\not\equiv0,\, I_\lambda'(v)=0\}.
\end{equation}

In the seminal work of Frank, Lenzmann, and Silvestre \cite{ FLS}, the equation \eqref{PDE} for $\lambda =0$ was examined. The authors demonstrated the uniqueness and nondegeneracy of positive ground state solutions $u \in H^{s}(\mathbb{R}^N)$ to the equation
\begin{equation}\label{eqFLS}
(- \Delta)^s u \, + \, u \, = \, |u|^{p-2}u \quad \mbox{in} \ \mathbb{R}^N,
\end{equation}
where 
\begin{equation*}
2 < p < p^{\star}(s, N) = 
\begin{cases}
\frac{N+2s}{N-2s}, \quad \mbox{if} \ 0 < 2s < N,\\
+\infty, \quad \ \ \mbox{if} \ 2s \geq N. 
\end{cases}
\end{equation*}
Among many other results, they showed  any positive solutions $u$ of \eqref{eqFLS} is radially symmetric, strictly decreasing and $u \in H^{2s+1}(\RR) \cap C^{\infty}(\RR)$ and satisfies the decay property 

\begin{equation}
\dfrac{C_1}{1 + |x|^{N+2s}}  \leq u(x) \leq \dfrac{C_2}{1 + |x|^{N+2s}},
\end{equation}
where $C_1$ and $C_2$ are constants depending on $N, p, s.$ We refer to the work of Frank and Lenzmann \cite{FL} for the dimension $N = 1.$ Concerning equation \eqref{eqFLS} for $s =1,$ it is well-known that positive solutions $\mathcal{V}$  are radially symmetric and unique up to translation and satisfies the following asymptotic estimate 
\begin{equation}\label{asymp}
\mathcal{V}(x) |x|^{\frac{(N-1)}{2}} e^{|x|} \rightarrow C, \quad \mbox{as} \ |x| \rightarrow \infty
\end{equation}
for some positive constant $C.$ Despite the apparent fact that \eqref{eqFLS} is the non-local counter of the scalar field equation, the latter case exhibits a number of novel and intriguing phenomena. An important question that remains intriguing is whether, in the non-local scenario, the positive solutions are unique. The uniqueness question has been addressed recently \cite{FW}, but an answer in full generality is still open. 
When $\lambda \neq 0,$ because of the presence of mixed type nonlinearities the solutions of \eqref{PDE} exhibits many new phenomenon. Let us first first demonstrate the existence results concerning \eqref{PDE}.

\begin{theorem}\label{Existence_theorem}
    Let $s\in(0,1),\, N>2s$ and one of the following condition holds:    \begin{enumerate}
        \item $p<2_s^\ast$ and $\lambda>0$; 
        \item $p=2_s^\ast$, $N\geq 4s$,  $q\in(2,2_s^\ast)$ and $\lambda>0$; 
        \item $p=2_s^\ast$, $2s<N<4s$,  $q\in(\frac{4s}{N-2s},2_s^\ast)$ and $\lambda>0$;
            \end{enumerate}
        Then Eq~\eqref{PDE} admits a positive radially symmetric and radially decreasing ground state solution $u_\lambda\in \H\cap C^\infty(\RR)$. 
\end{theorem}
The proof of the existence of radially symmetric and radially decreasing ground state solution $u_\lambda\in\H$ in Theorem~\ref{Existence_theorem} follows from \cite{A} when $p=2_s^\ast$ and from \cite{CW} when $p<2_s^\ast$. In Section~\ref{Prels}, we  discuss various qualitative properties of any positive weak solution of \eqref{PDE}.

\begin{remark}
The low dimensional phenomenon for the existence is a delicate question. In our forthcoming paper \cite{BDG}, we are working on the existence of ground state in $2s<N<4s$, $p=2_s^\ast$, $q\in(2,\frac{4s}{N-2s})$ and $\lambda>0$ is sufficiently large.
\end{remark}
This paper's main goal is to investigate the limiting behaviour of positive solutions of \eqref{PDE} as $\lambda \rightarrow 0$ and discuss the uniqueness and nondegenaracy issues. Since the problem is non-local in nature, it has lot of technical difficulties apart from some intrinsic characteristics the problem has which we shall describe along the subsequent results. In fact, we provide several novel steps to handle such non-linearities. To this end, let us first recall  the existence of positive solutions of fractional Emden-fowler equation :
\begin{equation}\label{Emden-Fowler}
\fra u=|u|^{2_s^\ast-2}u\quad\text{in }\RR.
\end{equation}
It is well known that all the positive solutions (up to translation and dilation) are given by \begin{equation}\label{Talenti}
U_1(x):=c_{N,s}\left(\frac{1}{1+|x|^2}\right)^{\frac{N-2s}{2}},
\end{equation} see \cite{CT}.
We define $U_\rho$, rescaling of $U_1$, as follows: 
\begin{equation}\label{Urho}
U_{\rho}(x):=\rho^{-\frac{N-2s}2}U_1\left(\frac{x}{\rho}\right), \quad\rho>0.
\end{equation}

Now we state the main theorems of this paper:

\begin{theorem}\label{Asymp_1}
  Let $N>4s,$ $p=2_s^\ast,$ $q\in(2,2_s^\ast)$ and $(u_\lambda)$ be a family of positive radially symmetric and radially decreasing ground state solutions of \eqref{PDE} with $\max_{x}u_\lambda(x)=u_\lambda(0)$. Then for small $\lambda>0$,
  \begin{align}
      &u_\lambda(0)\sim \lambda^{-\frac1{q-2}} \label{asymp-i}.\\
      &[u_\lambda]_{\H}^2\sim\|u_\lambda\|_{2_s^\ast}^{2_s^\ast}\sim1,\,\, \|u_\lambda\|_2^2\sim (2_s^\ast-q)\lambda^{\frac{2_s^\ast-2}{q-2}},\,\,\|u_\lambda\|_q^q\sim \lambda^{\frac{2_s^\ast-q}{q-2}}. \label{asymp-ii}
  \end{align}
  Further, as $\lambda\to0$, the rescaled family of ground states 
  \begin{equation}\label{rescale}
      v_{\lambda}(x)=\lambda^{\frac1{q-2}}u_\lambda\left(\lambda^{\frac{2_s^\ast-2}{2s(q-2)}}x\right),
  \end{equation} converges to $U_{\rho_0}$ in $\H\cap C^{2s-\delta}(\RR)$ where $\delta\in(0,2s)$ and
  \begin{equation}
      \rho_0=\left(\frac{2(2_s^\ast-q)\|U_1\|_q^q}{q(2_s^\ast-2)\|U_1\|_2^2}\right)^{\frac{2_s^\ast-2}{2s(q-2)}}.
  \end{equation}
 Further, the convergence rate is described by the relation \begin{equation}\label{conv_rate}
      [U_{\rho_0}]_{\H}^2-[v_\lambda]_{\H}^2\sim \lambda^\frac{2^*_s-2}{q-2}.
  \end{equation}
If  $2s<N\leq 4s$ and $\frac{4s}{N-2s}<q<2^*_s$ then there exists $\xi_\lambda\in(0,\infty)$ such that  $\xi_\lambda\to 0$ and
$$w_\lambda(x):=\xi_\lambda^\frac{N-2s}{2}v_\lambda(\xi_\lambda x),$$
converges to $U_1$ in $\HHH$. Moreover, the convergence rate is described by the relation
\begin{equation}
\lambda^\sigma\gtrsim  [U_1]^2_{H^s(\RR)}-[v_\lambda]^2_{H^s(\RR)}\gtrsim \begin{cases}
        \lambda^\sigma\left(\ln\frac1{\lambda}\right)^{-\frac{4-q}{q-2}}, \quad &\mbox{ if }\, N=4s\\
        \lambda^{\frac{2(N-2s)}{(N-2s)q-4s}},\quad & \mbox{ if }\,  2s<N<4s,\, q>\frac{4s}{N-2s}.
    \end{cases}
    \end{equation}  
  \end{theorem}
  
  \begin{remark}
It should be noted that we are unable to determine the precise convergence rate of the rescaled ground states when $2s < N \leq 4s$. As the Green's function of the operator $((-\Delta)^s + \lambda I)$ with $\lambda >0$ has polynomial decay in the non-local case (see \cite[Appendix~C]{FLS}), whereas it has exponential decay in the local case, which facilitates the estimation of some terms.

  \end{remark}

\begin{remark}
Note that since any ground state solution $u_\lambda$ of \eqref{PDE} is of the form $u_\lambda(.)=u_\lambda(|.-x_\lambda|)$, for some $x_\lambda\in\RR$ and \eqref{PDE} is invariant under translation , w.l.g. we can assume $\max_{x}u_\lambda(x)=u_\lambda(0)$.
\end{remark}

For the subcritical case $p<2_s^\ast$, we prove that 
\begin{theorem}\label{Asymp_2}
    Let $N>2s,\,2<q<p<2_s^\ast$ and $(u_\lambda)$ be a family of positive radially symmetric and radially decreasing ground state solutions of \eqref{PDE} with $\max_{x}u_\lambda(x)=u_\lambda(0)$. Then for $\lambda$ small,
    \begin{align*}
    \|u_0\|_p^p-\|u_\lambda\|_p^p&\sim\lambda.\\
    \|u_0\|_2^2-\|u_\lambda\|_2^2&\sim\lambda.\\
    [u_0]_{\H}^2-[u_\lambda]_{\H}^2&\sim\lambda.
\end{align*}
Moreover, $u_\lambda$ converges to the unique positive, radially symmetric and radially decreasing ground state solution $u_0$ of 
\begin{equation}
\label{subcri-uni-sol}\fra u +u=|u|^{p-2}u
\end{equation} in $\H\cap C^{2s-\delta}(\RR)$ where $\delta\in(0,2s)$.
\end{theorem}

In general, the uniqueness of ground state solution to \eqref{PDE} for all $\lambda>0$ is not expected. Indeed, in the case $s=1$ and $p<2^*$, D\'avila, del Pino and Guerra \cite{DDG} constructed multiple positive solutions to \eqref{PDE} for a sufficiently large $\lambda$. Also, in the same paper they have exhibited numerical simulation  suggesting the nonuniqueness in the critical case $p=2^*$. Wei and Wu \cite{WW} recently proved that there exist
two positive solutions to \eqref{PDE} when $s=1$, $N=3$, $p=2^*$, $q\in(2, 4)$ and $\lambda>0$ is sufficiently large, as suggested in \cite{DDG}. Also see \cite{WW1}, where existence and qualitative properties of normalized solutions were established for equations with the critical Sobolev and mixed nonlinearities. 

Here we prove uniqueness and nondegenracy of ground state to \eqref{PDE} when $\lambda>0$ is sufficiently small in the nonlocal case  $s\in (0,1)$ (both in the cases when $p$ is critical and subcritical). For the local case $s=1$, uniqueness and nondegenracy of ground states for $\lambda>0$ sufficiently small has been proven by Akahori, et al. in \cite{AIIKN}. 

\begin{theorem}\label{t:uni-cri}
	Let $2+\frac{4s}{N}<q<p=2_s^\ast$ and $N>4s$. Then there exists $\lambda_0>0$ such that for any $\lambda\in(0,\lambda_0)$, the positive ground state to \eqref{PDE} is unique. Moreover, the unique positive ground state  solution is nondegenerate in $H^s_{rad}(\RR)$.
\end{theorem}

\begin{theorem}\label{t:uni-sub}
	Let  $2<q<p<2_s^\ast$ and  $N>2s$.  Then there exists $\lambda_0>0$ such that for any $\lambda\in(0,\lambda_0)$, the positive ground state to \eqref{PDE} is unique.  Moreover, the unique positive ground state  solution is nondegenerate in $H^s_{rad}(\RR)$.
\end{theorem}
The rest of the paper is organized as follows. Section~\ref{Prels} deals with the preliminaries and we discuss some qualitative properties of the solutions of \eqref{PDE}. Section~\ref{critical} and \ref{subcritical} are devoted to the proof of Theorem~\ref{Asymp_1} and Theorem~\ref{Asymp_2} respectively. In Section~\ref{loc_uniq}, we give the proof of Theorem~\ref{t:uni-cri} and Theorem~\ref{t:uni-sub} respectively.

\vspace{5mm}

{\bf Notations:} Throughout this paper, we denote by $C$ the  generic positive constant which may vary from line to line.  For $\lambda\gg1$ and $f(\lambda),\, g(\lambda) \geq 0,$ whenever there exists $\lambda_0>0$ such that for every $0<\lambda \leq \lambda_0$ the respective
condition holds:
\begin{itemize}
	\item
	$f(\lambda)\lesssim g(\lambda)$ if there exists $C>0$ independent of $\lambda$ such that $f(\lambda) \leq Cg(\lambda);$
	\item
	$f(\lambda)\sim g(\lambda)$ if $f(\lambda) \lesssim g(\lambda)$ and $g(\lambda) \lesssim f(\lambda);$
	\item
	$f(\lambda) \simeq g(\lambda)$ if $f(\lambda) \sim g(\lambda)$ and $\lim_{\lambda \to 0}\frac{f(\lambda)}{g(\lambda)}=1.$
\end{itemize}
We denote by $\|.\|_p$, the $L^p$ norm in $\RR$.
We also use the standard notations $f=O(g)$ and $f=o(g),$ where $f \geq 0,\, g \geq 0$.

\section{Preliminaries}\label{Prels}
If $u_\lambda$ is any positive weak solution of \eqref{PDE} with $2<q<p\leq2_s^\ast$, then following the arguments as in 
 \cite[Proposition~5.1.1]{DMV} yields  $u_\lambda\in L^\infty(\RR)$. Before proving some qualitative properties of $u_\lambda$, we first recall a lemma.
\begin{lemma}[Lemma C.2, \cite{FLS}]\label{FLS_theorem}
     Let $N > 1$, $s\in(0,1)$, and suppose that $V\in L^\infty(\RR)$ with $V (x)\to 0$ as
$|x|\to\infty$. Assume that $u\in L^2(\RR)$ with $\|u\|_2 = 1$ satisfies $\fra u + V u = Eu$ with
some $E < 0$. Furthermore, let $0 < \beta < -E$ be given and suppose that $R > 0$ is such that
$V (x) +\beta\geq 0$ for $|x| \geq R$. Then the following properties hold:
\begin{enumerate}
    \item $|u(x)|\lesssim (1+|x|)^{-(N+2s)}$,
    \item $u(x)=-c\beta^{-2}\left(\int_{\RR}Vu\,dx\right)|x|^{-(N+2s)}+o(^{-(N+2s)})$ as $|x|\to\infty$
\end{enumerate}
where $c>0$ is a positive constant depending only on $N$ as $s.$
\end{lemma}
\begin{theorem}
    If $u$ is a positive weak solution of \eqref{PDE} then $u\in C^\infty(\RR).$ Further, $$u(x)=C|x|^{-(N+2s)}+o(|x|^{-(N+2s)})$$ as $|x|\to\infty.$
\end{theorem}
\begin{proof}
  $u\in L^\infty(\RR)$ implies $u^{p-1}+u^{q-1}-u=f(u)\in L^\infty(\RR).$ Therefore, applying the Scha\"uder estimate \cite[Theorem 1.1(a)]{RS2} we get $u\in C^{2s}(B_{1/2}(0))$ when $s\neq\frac12$ and $u\in C^{2s-\epsilon}(B_{1/2}(0))$ when $s=\frac12$.
Moreover, since the equation is invariant under translation, translating the equation we obtain  $u\in C^{2s}(\RR)$ when $s\neq\frac12$ and $u\in C^{2s-\epsilon}(\RR)$ when $s=\frac12$.
 for every $\epsilon>0.$  Thus applying \cite[Theorem 1.1(b)]{RS2} we get $u\in C^{2s+\alpha}_{\text{loc}}(\RR)$ 
for some $\alpha\in(0,1).$ Next it can be shown exactly as in \cite[Proposition 1]{BM1} that $u$ is a classical 
solution of \eqref{PDE}. Further, as 
$u\in C^{2s+\alpha}_{\text{loc}}
(\RR)$, repeatedly applying 
Scha\"uder estimate as in \cite[Corollary 
2.4]{RS1} we get $u\in 
C^{\infty}(\RR).$ Moreover, $u\in 
C^{\alpha}(\RR)\cap L^2(\RR)$ implies
$u(x)\to 0$ as $|x|\to\infty$. To 
prove the asymptotic estimate of 
$u$ as $|x|\to\infty$, we use the 
Lemma~\ref{FLS_theorem}. Note that 
$\Tilde{u}:=\frac{u}{\|u\|_2}$ 
satisfies $$\fra 
\Tilde{u}+V\Tilde{u}=-\Tilde{u},$$ 
where  $$V(x)=-\|u\|_2^{p-
2}\Tilde{u}^{p-2}(x)-\lambda 
\|u\|_2^{q-2}\Tilde{u}^{q-2}
(x)=-u^{p-2}(x)-\lambda u^{q-2}
(x).$$ Since $u\in L^{\infty}(\RR)$, $V\in L^{\infty}(\RR)$. Further, $V(x)\to0$ as $|x|\to\infty$. Let $0<\beta<1$, there exists $R>0$ such that $V(x)+\beta\geq0$ for all $|x|\geq R.$ Therefore, by Lemma \ref{FLS_theorem}(1), $$\Tilde{u}(x)\leq C_1|x|^{-(N+2s)}.$$
Moreover, $$-c\beta^{-2}\int_{\RR}V(x)\Tilde{u}(x)\,dx=\frac{c}{\beta^2\|u\|_2}\int_{\RR}(u^{p-1}(x)+\lambda u^{q-1}(x))\,dx>0.$$
Thus by Lemma \ref{FLS_theorem}(2), $u(x)=C|x|^{-(N+2s)}+o(|x|^{-(N+2s)})$ as $|x|\to\infty.$ 
\end{proof}

\begin{lemma}\label{l:rad}
Let $u\in H^s(\RR)$ be a positive weak solution of \eqref{PDE}.	
Then $u$ is radially symmetric and strictly decreasing about some point in $\RR$.
\end{lemma}
\begin{proof}
Following the arguments as in \cite[Theorem 1.2]{BM2}, proof of this theorem follows. 
\end{proof}

\begin{theorem}({\bf Pohozaev identity})[Theorem A.1, \cite{BM1}]\label{Pohozaev_theorem}
    Let $u\in \HHH\cap L^{\infty}(\RR)$ be a positive solution of $$\fra u=f(u)\quad\text{in }\RR,$$ and $F(u)=\int_0^uf(t)\,dt\in L^1(\RR)$. Then $$(N-2s)\int_{\RR}uf(u)\,dx=2N\int_{\RR}F(u)\,dx.$$
\end{theorem}
\begin{corollary}\label{Pohozaev_corollary}
    Let $u$ be a positive solution of \eqref{PDE}, then 
    \begin{equation*}
        (N-2s)[u]_{\HHH}^2=2N\left(\frac1{p}\int_{\RR}|u|^{p}\,dx+\frac{\lambda}{q}\int_{\RR}|u|^q\,dx-\frac12\int_{\RR}|u|^2\,dx\right).
    \end{equation*}
\end{corollary}

\noindent\textbf{Palais-Smale decomposition:}
We denote by $J_0(u)$ the energy functional corresponding to the critical Emden-Fowler equation 
\begin{equation}\label{limit_equation}
    \fra u= |u|^{2_s^\ast-2}u.
\end{equation}
Therefore, $J_0=\frac12[u]_{\H}^2-\frac1{2_s^\ast}\|u\|_{2_s^\ast}^{2_s^\ast}$. Next, we recall the Palais-Smale decomposition result for $J_0$.
\begin{theorem}[Theorem 1.1, \cite{PP2}]\label{PS_theorem_2}
    Let $w_\lambda$ be a PS sequence for $J_0$ in $\HHH$. Then there exists a solution $W\in\HHH$ to \eqref{limit_equation} such that for a subsequence, still denoted by $w_\lambda$, $w_\lambda\rightharpoonup W$. Moreover either the convergence is strong, or there exists a natural number $k$ and nontrivial solutions $w^{(j)}\in\HHH$ to \eqref{limit_equation} and sequences $R_{\lambda}^{(j)}\in(0,\infty)$, $x^{(j)}\in\RR$ for $1\leq j\leq k$ such that 
    \begin{equation*}
        \left|\log\left(\frac{R_{\lambda}^{(i)}}{R_{\lambda}^{(j)}}\right)\right|+\left|\frac{x_{\lambda}^{(i)}-x_{\lambda}^{(j)}}{R_{\lambda}^{(i)}}\right|\to\infty
\quad\text{as }\lambda\to 0,\quad\text{for }i\neq j
    \end{equation*}
    and 
    \begin{equation*}
    \begin{aligned}
w_{\lambda}(x)=W(x)+\sum_{j=1}^k (R_{\lambda}^{(j)})^{-\frac{N-2s}2}w^{(j)}\left(\frac{x-x_{\lambda}^{(j)}}{R_{\lambda}^{(j)}}\right)+o(1)\quad\text{in }\HHH\\
\|w_{\lambda}\|_{\HHH}^2=\|W\|_{\HHH}^2+\sum_{j=1}^k \|w^{(j)}\|_{\HHH}^2+o(1)\quad\text{as }\lambda\to 0.
        \end{aligned}
    \end{equation*}
    Further, \[
J_0(w_\lambda)=J_0(W)+\sum_{j=1}^k J_0(w^{(j)})+o(1)\quad\text{as }\lambda\to0.\]

Moreover if $\{w_n\}$ is a positive PS sequence then $W,\, w^j$, $j=1,\ldots, k$ are nonnegative solutions to  \eqref{limit_equation}.
\end{theorem}
We denote the space of all radial functions in $\H$ by $H_r^s(\RR)$ (or by $H^s_{rad}(\RR)$) and mention the Strauss compactness lemma which states that
\begin{lemma}[Theorem 7.1, \cite{ND}]\label{Strauss_compactness_lemma}
    For $2<q<2_s^\ast$, the space $H_r^s(\RR)$ is compactly embedded in $L^q(\RR)$.
\end{lemma}
We'll end this section by mentioning the Mountain Pass characterization of the least energy solution from \cite{A}. The same proof works for $p<2_s^\ast$.
Let's define $\Gamma=\{\gamma\in C([0,1];\H): \gamma(0)=0,I_\lambda(\gamma(1))<0\}$ and $c_\lambda:=\inf_{\gamma\in\Gamma}\max_{t\in[0,1]}I_\lambda(\gamma(t))$.
\begin{theorem}[Theorem 2, \cite{A}]\label{MPC_of_GST}
    Let $u$ be the ground state solution of \eqref{PDE} obtained in Theorem \ref{Existence_theorem}. Then there exists a path $\gamma\in\Gamma$ such that $u\in\gamma([0,1])$ and $\max_{t\in[0,1]}I(\gamma(t))=I(u).$ Moreover, $$m^\lambda=c_\lambda=\inf_{\gamma\in\Gamma}\max_{t\in[0,1]}I_\lambda(\gamma(t)).$$
\end{theorem}

\section{Asymptotic behaviour when $p=2_s^\ast$}\label{critical}

This section is devoted to study the asymptotic behaviour of positive solutions to \eqref{PDE} when $\lambda \rightarrow 0$
and $p=2_s^\ast.$  To this end, let us recall
\begin{equation*}
        I_{\lambda}(u)=\frac12[u]_{\H}^2+\frac12\int_{\RR}|u|^2\,dx-\frac1{2_s^\ast}\int_{\RR}|u|^{2_s^\ast}\,dx-\frac{\lambda}{q}\int_{\RR}|u|^q\,dx
\end{equation*}
is the energy functional associated to \eqref{PDE}.
For $\lambda=0$, the PDE becomes 
\begin{equation}\label{limit_equation_2}
    \fra u+u=|u|^{2_s^\ast-2}u
\end{equation} and define the energy functional associated to it by 
\begin{equation}
    I_0(u):=\frac12[u]_{\H}^2+\frac12\|u\|_2^2-\frac1{2_s^\ast}\|u\|_{2_s^\ast}^{2_s^\ast}.
\end{equation}
It's well known that the above PDE doesn't have any nontrivial finite energy solution. It can easily be proved by the Pohozaev identity. We define the Nehari manifolds associated to $I_\lambda$ and $I_0$ by 
\begin{align}
    \MM_{\lambda}&:=\{u\in\H\setminus\{0\}: [u]_{\H}^2+\|u\|_2^2=\|u\|_{2_s^\ast}^{2_s^\ast}+\lambda\|u\|_q^q\}.\label{M_lambda}\\
    \MM_0&:=\{u\in\H\setminus\{0\}: [u]_{\H}^2+\|u\|_2^2=\|u\|_{2_s^\ast}^{2_s^\ast}\}.\label{M_0}
\end{align}
And for all $\lambda\geq0$, we define $$m_\lambda^\ast:=\inf_{u\in\MM_\lambda}I_\lambda(u),\quad c_\lambda':=\inf_{u\in\H\setminus\{0\}}\max_{t\geq0}I_\lambda(tu),\quad c_\lambda=\inf_{\gamma\in\Gamma}\max_{t\in[0,1]}I_{\lambda}(\gamma(t)),$$
where $\Gamma=\{\gamma\in C([0,1],\H):\gamma(0)=0,I_{\lambda}(\gamma(1))<0\}.$
\begin{remark}
    For all $\lambda\geq0$, for any nonzero $ u\in\H$ there exists a unique $t(u)>0$ such that $t(u)u\in\MM_\lambda$ and $$\max_{[0,\infty)}I_\lambda(tu)=I_\lambda(t(u)u).$$
    For a proof see chapter 4 of \cite{W}. 
\end{remark}

Now we begin with some preparatory lemmas which are needed to establish the asymptotic profiles of the solutions.

\begin{lemma}\label{M_lambda_char}
For any $\lambda\geq0$, $m_\lambda^\ast=c_\lambda'\geq c_\lambda.$
\end{lemma} 
\begin{proof}\phantom{\qedhere}
$$m_\lambda^\ast=\inf_{u\in\MM_\lambda}I_\lambda(u)=\inf_{u\in\MM_\lambda}\max_{t\geq0}I_\lambda(tu)\geq c_\lambda'.$$
    For the reverse inequality observe that for any $u\in\H\setminus\{0\}$, $\max_{t\geq0}I_\lambda(tu)=I_\lambda(t(u)u)\geq m_\lambda^\ast$.

    Let $u\in\H\setminus\{0\}$.  It's easy to observe that $I_\lambda(tu)>0$ for small $t>0$ and $I_\lambda(tu)<0$ for large $t>0$. Thus it's possible to choose a reparametrization $\gamma_u:[0,1]\to\H$ of the ray passing through $0$ and $u$ and satisfying the relation $I_\lambda(\gamma_u(1))<0.$ Thus $\forall u\in\H\setminus\{0\}$, 
    \begin{equation*}
\max_{t\geq0}I_\lambda(tu)=\max_{t\in[0,1]}I_\lambda(\gamma_u(t))\geq c_\lambda.\tag*{\qed}
    \end{equation*}
\end{proof}

\begin{lemma}\label{M_lambda_GS}
    For any $\lambda>0,$ $m^\lambda=m_\lambda^\ast$. In particular,  $\inf_{u\in\MM_\lambda}I_\lambda(u)=I_\lambda(u_\lambda)$ .
\end{lemma}
\begin{proof}\phantom{\qedhere}
    By the definition of Nehari manifold, any nonzero critical point of $I_\lambda$ lies in $\MM_\lambda$. Thus, set of all nonzero critical points is a subset of $\MM_\lambda$ and $$m_\lambda^\ast\leq m^\lambda.$$
    For the reverse direction, we combine Lemma \ref{M_lambda_char} and Theorem \ref{MPC_of_GST}.
    \begin{equation*}
        m_{\lambda}^\ast\geq c_\lambda=m^\lambda.\tag*{\qed}
    \end{equation*}
\end{proof}
\begin{lemma}\label{u_lambda_bdd}
    The family $(u_\lambda)$ is uniformly bounded in $\H$.
\end{lemma}
\begin{proof}
    Observe that $I_0(u)\geq I_\lambda(u)$ for any $u\in\H$. Thus by Lemma \ref{M_lambda_char}, $$m_0^\ast=\inf_{u\in\H\setminus\{0\}}\max_{t\geq0}I_0(tu)\geq \inf_{u\in\H\setminus\{0\}}\max_{t\geq0}I_\lambda(tu)=m_\lambda^\ast.$$
    Using, Lemma \ref{M_lambda_GS},
    \begin{equation*}
            m_0^\ast\geq m_\lambda^\ast=I_{\lambda}(u_\lambda)-\frac1q\langle I_{\lambda}'(u_\lambda),u_\lambda\rangle\geq \left(\frac12-\frac1q\right)\|u\|_{\H}^2.
    \end{equation*}
    Thus, $(u_\lambda)$ is uniformly bounded in $\H$.
\end{proof}
For $\lambda>0$, we recall the rescaling defined in \eqref{rescale}
$$
    v(x)=\lambda^{\frac1{q-2}}u\left(\lambda^{\frac{2_s^\ast-2}{2s(q-2)}}x\right)
$$ which transforms \eqref{PDE} into 
\begin{equation}\label{Rescaled_PDE}
    \fra v+\lambda^\sigma v=|v|^{2_s^\ast-2}v+\lambda^\sigma |v|^{q-2}v\quad\text{in }\RR,
\end{equation}
where 
\begin{equation}
    \sigma=\frac{2_s^\ast-2}{q-2}.
\end{equation}
The corresponding energy functional is given by 
\begin{equation}
J_\lambda(v)=\frac12[v]_{\H}^2+\frac{\lambda^\sigma}2\int_{\RR}|v|^2\,dx-\frac1{2_s^\ast}\int_{\RR}|v|^{2_s^\ast}\,dx-\frac{\lambda^\sigma}{q}\int_{\RR}|v|^q\,dx.
\end{equation}
The formal limit equation for \eqref{Rescaled_PDE} as $\lambda\to0$ is the Emden-Fowler equation \eqref{limit_equation}.
We denote the corresponding Nehari manifolds as follows:
\begin{align}
    \NN_{\lambda}&:=\{v\in\H\setminus\{0\}: [v]_{\H}^2+\lambda^\sigma\|v\|_2^2=\|v\|_{2_s^\ast}^{2_s^\ast}+\lambda^\sigma\|v\|_q^q\}\label{N_lambda}.\\
    \NN_0&:=\{v\in\H\setminus\{0\}: [v]_{\H}^2=\|v\|_{2_s^\ast}^{2_s^\ast}\}\label{N_0}.
\end{align}
And define, $$m_\lambda:=\inf_{u\in\NN_\lambda}J_\lambda(u),\quad m_0:=\inf_{u\in\NN_0}J_0(u).$$
\begin{lemma}\label{rescaling_lemma}
    For $\lambda>0$, $u\in\H$ and $v$ is the rescaling \eqref{rescale} of $u.$ Then
    \begin{enumerate}
        \item $[v]_{\H}^2=[u]_{\H}^2$, $\|v\|_{2_s^\ast}^{2_s^\ast}=\|u\|_{2_s^\ast}^{2_s^\ast}.$
        \item $\lambda^\sigma \|v\|_2^2=\|u\|_2^2$, $\lambda^\sigma \|v\|_q^q=\lambda\|u\|_q^q.$
        \item $J_\lambda(v)=I_\lambda(u)$, $m_\lambda=m_\lambda^\ast$.
    \end{enumerate}
    In particular, $m_\lambda=\inf_{u\in\mathcal{N}_\lambda}J_\lambda(u)=J_\lambda(v_\lambda)$.
\end{lemma}
\begin{proof}
    \noindent{(1)} $$\|v\|_{2_s^\ast}^{2_s^\ast}=\lambda^{\frac{2_s^\ast}{q-2}}\int_{\RR}\left|u\left(\lambda^{\frac{2_s^\ast-2}{2s(q-2)}}x\right)\right|^{2_s^\ast}\,dx=\lambda^{\frac{2_s^\ast}{q-2}}\lambda^{\frac{-N(2_s^\ast-2)}{2s(q-2)}}\|u\|_{2_s^\ast}^{2_s^\ast}=\|u\|_{2_s^\ast}^{2_s^\ast}.$$
    and 
    \begin{equation*}
        \begin{aligned}
            [v]_{\H}^2&=\lambda^{\frac{2}{q-2}}\iint_{\mathbb{R}^{2N}}\frac{|u(\lambda^{\frac{2_s^\ast-2}{2s(q-2)}}x)-u(\lambda^{\frac{2_s^\ast-2}{2s(q-2)}}y)|^2}{|x-y|^{N+2s}}\,dx\,dy\\&=\lambda^{\frac{2}{q-2}}\lambda^{\frac{-(N-2s)(2_s^\ast-2)}{2s(q-2)}}[u]_{\H}^2=[u]_{\H}^2.
        \end{aligned}
    \end{equation*}
    \noindent{(2)}
    \begin{equation*}
        \begin{aligned}
            \lambda^\sigma \|v\|_2^2&=\lambda^{\frac{2_s^\ast-2}{q-2}}\lambda^{\frac{2}{q-2}}\lambda^{\frac{-N(2_s^\ast-2)}{2s(q-2)}}\|u\|_2^2=\|u\|_2^2.\\
            \lambda^\sigma \|v\|_q^q&=\lambda^{\frac{2_s^\ast-2}{q-2}}\lambda^{\frac{q}{q-2}}\lambda^{\frac{-N(2_s^\ast-2)}{2s(q-2)}}\|u\|_q^q=\lambda\|u\|_q^q.
        \end{aligned}
    \end{equation*}
    \noindent{(3)} By (1), (2), $J_\lambda(v)=I_\lambda(u)$. Since scaling \eqref{rescale} gives a bijection between the Nehari manifolds $\MM_\lambda$ and $\NN_\lambda$, $m_\lambda=m_\lambda^\ast.$
\end{proof}
Suppose $v_\lambda$ is the rescaling \eqref{rescale} of $u_\lambda$. Then by the above lemma $J_\lambda(v_\lambda)=I_\lambda(u_\lambda)$ and $v_\lambda$ is a ground state solution of \eqref{Rescaled_PDE}. By Corollary \ref{Pohozaev_corollary}, $v_\lambda$ satisfies the Pohozaev identity, 
\begin{equation}
    \frac1{2_s^\ast}[v_\lambda]_{\H}^2+\frac{\lambda^\sigma}{2}\|v_\lambda\|_2^2=\frac1{2_s^\ast}\|v_\lambda\|_{2_s^\ast}^{2_s^\ast}+\frac{\lambda^\sigma}{q}\|v_\lambda\|_q^q.
\end{equation}
Note that $m_0$ is attained on $\NN_0$ by the Talenti function $$U_1(x)=c_{N,s}\frac1{(1+|x|^2)^{\frac{N-2s}{2}}}$$ and it's family of rescalings $$U_\rho(x)=\rho^{-\frac{N-2s}{2}}U_1(x/\rho),\rho>0.$$ For $v\in \H\setminus\{0\}$, we set
\begin{equation}\label{tau_def}
    \tau(v):=\frac{[v]_{\H}^2}{\|v\|_{2_s^\ast}^{2_s^\ast}}.
\end{equation}
 Then $(\tau(v))^{\frac{N-2s}{4s}}v\in\NN_0$ for all $v\in \H\setminus\{0\}$ and $v\in\NN_0$ if{}f $\tau(v)=1.$
\begin{lemma}\label{N_lambda_char}
    Let $\lambda\geq0$. Set 
    \begin{equation}\label{4g}
        v_t(x)=\begin{cases}
            v\left(\frac xt\right),\quad t>0\\0,\quad t=0.
        \end{cases}
    \end{equation}
Then $$m_\lambda=\inf_{v\in\H\setminus\{0\}}\sup_{t\geq0}J_\lambda(tv)=\inf_{v\in\H\setminus\{0\}}\sup_{t\geq0}J_\lambda(v_t).$$ In particular, we have $m_\lambda=J_\lambda(v_\lambda)=\sup_{t\geq0}J_\lambda(tv_\lambda)=\sup_{t\geq0}J_\lambda((v_\lambda)_t)$.
\end{lemma}
\begin{proof}
Following the arguments of Lemma \ref{M_lambda_GS}, it's easy to see that $$m_\lambda=\inf_{v\in\H\setminus\{0\}}\sup_{t\geq0}J_\lambda(tv).$$ 
Next we prove that $m_\lambda=\inf_{v\in\H\setminus\{0\}}\sup_{t\geq0}J_\lambda(v_t)$.  Towards this, we first note that

 (i) for any $v\in\H$ with $$\int_{\RR}|v|^{2_s^\ast}\,dx+\lambda^\sigma\int_{\RR}|v|^q\,dx-\lambda^\sigma\int_{\RR}|v|^2\,dx>0, $$ there exists a unique 
$$t(v):=\bigg(\frac{[v]_{\H}^2}{\|v\|^{2_s^\ast}_{2_s^\ast}+\lambda^\sigma\|v\|^{q}_{q}-\lambda^\sigma\|v\|^{2}_{2}}\bigg)^\frac{1}{2s} >0$$ such that $v_{t(v)}\in\NN_\lambda$. As $v_\lambda\in\mathcal{N}_\lambda$, it follows that $t(v_\lambda)=1$.

(ii) For any $v\in\H$ with $$\frac1{2_s^\ast}\int_{\RR}|v|^{2_s^\ast}\,dx+\frac{\lambda^\sigma}q\int_{\RR}|v|^q\,dx-\frac{\lambda^\sigma}2\int_{\RR}|v|^2\,dx>0 $$ there exists a unique $$\tilde{t}(v)^{2s}:=\frac{[v]_{\H}^2}{2_s^\ast\left(\frac{\|v\|^{2_s^\ast}_{2_s^\ast}}{2_s^\ast}+\lambda^\sigma\frac{\|v\|_q^q}{q}-\lambda^\sigma\frac{\|v\|_2^2}{2}\right)} >0$$ such that $\max_{t\geq 0}J_\lambda(v_t)=J_\lambda(v_{\tilde{t}(v)})$. Using the Pohozaev identity for $v_\lambda$, it follows that $\tilde{t}(v_\lambda)=1$.

   Therefore,  $$\inf_{v\in\H\setminus\{0\}}\sup_{t\geq0}J_\lambda(v_t)\leq \sup_{t\geq0}J_\lambda((v_\lambda)_t)=J_\lambda(v_\lambda)=m_\lambda.$$
   
    For the reverse inequality, we see that if $\frac{\lambda^\sigma}2\int_{\RR}|v|^2\,dx-\frac1{2_s^\ast}\int_{\RR}|v|^{2_s^\ast}\,dx-\frac{\lambda^\sigma}q\int_{\RR}|v|^q\,dx\geq0$ then 
    $$\sup_{t\geq0} J_\lambda(v_t)=\sup_{t\geq0} \big(c_1t^{N-2s}+c_2t^N\big),$$
 where $c_1=c_1(v)=\frac{1}{2}[v]_{\H}^2$  and $c_2=c_2(v)=\frac{\lambda^\sigma}2\|v\|_2^2-\frac1{2_s^\ast}\|v\|_{2_s^\ast}^{2_s^\ast}-\frac{\lambda^\sigma}q\|v\|_q^q$. Since $c_1, c_2\geq 0$, clearly 
 $\sup_{t\geq0} J_\lambda(v_t)=\infty>m_\lambda.$ On the other hand, if $\frac{\lambda^\sigma}2\int_{\RR}|v|^2\,dx-\frac1{2_s^\ast}\int_{\RR}|v|^{2_s^\ast}\,dx-\frac{\lambda^\sigma}q\int_{\RR}|v|^q\,dx<0 $ then a simple computation yields $\int_{\RR}|v|^{2_s^\ast}\,dx+\lambda^\sigma\int_{\RR}|v|^q\,dx-\lambda^\sigma\int_{\RR}|v|^2\,dx>0$ and hence by (i), $v_{t(v)}\in \mathcal{N}_\lambda$. 
 Therefore, 
 $$m_\lambda=\inf_{v\in\mathcal{N}_\lambda}J_\lambda(v)\leq J_\lambda(v_{t(v)})\leq \sup_{t\geq0}J_\lambda(v_t).$$
 Thus for any nonzero $v\in\H$, $\sup_{t\geq0}J_\lambda(v_t)\geq m_\lambda$. Hence the lemma follows. 
 \end{proof}
\begin{lemma}\label{v_lambda_bdd}
    Let $\lambda>0$. The family $\{v_\lambda\}$ is uniformly bounded in $\H$. In particular, $\{v_\lambda\}$ is bounded in $L^p(\RR)$ uniformly for all $p\in[2,2_s^\ast].$
\end{lemma}
\begin{proof}
    Since $\{u_\lambda\}$ is uniformly bounded in $\H$ and $[u_\lambda]_{\H}=[v_\lambda]_{\H}$ so we just need to show $\{v_\lambda\}$ is uniformly bounded in $L^2(\RR).$ By $v_\lambda\in\NN_\lambda$ and Corollary \ref{Pohozaev_corollary},
    \begin{align*}
        [v_\lambda]_{\H}^2+\lambda^\sigma\|v_\lambda\|_2^2&=\|v_\lambda\|_{2_s^\ast}^{2_s^\ast}+\lambda^\sigma\|v_\lambda\|_q^q.\\
        \frac1{2_s^\ast}[v_\lambda]_{\H}^2+\frac{\lambda^\sigma}{2}\|v_\lambda\|_2^2&=\frac1{2_s^\ast}\|v\|_{2_s^\ast}^{2_s^\ast}+\frac{\lambda^\sigma}{q}\|v_\lambda\|_q^q.
    \end{align*}
    Thus, 
    \begin{align}
        \left(\frac12-\frac1{2_s^\ast}\right)\lambda^\sigma\|v_\lambda\|_2^2&=\left(\frac1q-\frac1{2_s^\ast}\right)\lambda^\sigma\|v_\lambda\|_q^q.\\
        \|v_\lambda\|_2^2&=\frac{2(2_s^\ast-q)}{q(2_s^\ast-2)}\|v_\lambda\|_q^q\label{norm_2,q}.
    \end{align}
    Using Sobolev Embedding theorem and interpolation inequality,
\begin{equation}
    \|v_\lambda\|_2^2\leq \left(\frac{2(2_s^\ast-q)}{q(2_s^\ast-2)}\right)^{\frac{2_s^\ast-2}{q-2}}\left(\frac1{S_\ast}[v_\lambda]_{\H}^2\right)^{\frac{2_s^\ast}{2}}.
\end{equation}
Therefore, $\{v_\lambda\}$ is bounded in $L^2(\RR).$ By interpolation inequality for any $p\in[2,2_s^\ast]$,
\begin{equation*}
    \|v_\lambda\|_p^p\leq \left(\|v_\lambda\|_2\right)^{\frac{2(2_s^\ast-p)}{2_s^\ast-2}}\left(\|v_\lambda\|_{2_s^\ast}\right)^{\frac{2_s^\ast(p-2)}{2_s^\ast-2}}\leq \left(\frac{2(2_s^\ast-q)}{q(2_s^\ast-2)}\right)^{\frac{2_s^\ast-p}{q-2}}\left(\frac1{S_\ast}[v_\lambda]_{\H}^2\right)^{\frac{2_s^\ast}{2}}.
\end{equation*}
and we have (see \cite{MM1})$$\lim_{q\to2}\left(\frac{2(2_s^\ast-q)}{q(2_s^\ast-2)}\right)^{\frac{2_s^\ast-p}{q-2}}=e^{-\frac{N(2_s^\ast-p)}{4}},\quad\forall p\in[2,2_s^\ast].$$ Therefore, $\{v_\lambda\}$ is bounded in $L^p(\RR)$ uniformly for all $p\in[2,2_s^\ast].$
\end{proof}
\noindent\textbf{Idea of the proof of theorem \ref{Asymp_1}:} Since $\{v_\lambda\}$ is a family of positive radially symmetric and decreasing functions bounded in $\H$, there exists $v_0\in\H$ such that, by Lemma \ref{Strauss_compactness_lemma},
\begin{equation}
    v_\lambda\rightharpoonup v_0\quad\text{in }\H,\quad v_\lambda\to v_0\quad\text{in } L^p(\RR)\quad \forall p\in(2,2_s^\ast)
\end{equation}
and 
\begin{equation}
    v_\lambda(x)\to v_0(x)\quad \text{a.e. }\RR, v_\lambda\to v_0\quad\text{in } L_{\text{loc}}^2(\RR).
\end{equation}
Therefore, passing the limit $\lambda\to 0$ in the weak formulation of \eqref{Rescaled_PDE} yields that $v_0$ satisfies $\fra v=v^{2^*_s-1}$ in $\RR$. Next we shall show that $\{v_\lambda\}$ is a PS sequence for $J_0$. To do that we need to find an estimate on $m_\lambda-m_0$. Next using the PS decomposition theorem \ref{PS_theorem_2}, we'll show that $v_\lambda\to v_0$ in $\HHH.$
Finally we'll show that $v_\lambda\to v_0$ in $L^2(\RR).$
\begin{remark}[Remark 4.5, \cite{MM1}]\label{remark_estimate}
It's easy to check that 
\begin{equation*}
    \lim_{q\to2}\left(\frac2q\right)^{\frac{2_s^\ast-2}{q-2}}=e^{\frac{-2s}{N-2s}},\quad \lim_{q\to2}\left(\frac{2_s^\ast-q}{2_s^\ast-2}\right)^{\frac{2_s^\ast-2}{q-2}}=e^{-1},
\end{equation*}
and $$\lim_{q\to2_s^\ast}\frac1{2_s^\ast-q}\left(\frac{2_s^\ast-q}{2_s^\ast-2}\right)^{\frac{2_s^\ast-2}{q-2}}=\frac{N-2s}{4s}=\frac1{2_s^\ast-2}.$$ Hence, 
$$\left(\frac{2(2_s^\ast-q)}{q(2_s^\ast-2)}\right)^{\frac{2_s^\ast-2}{q-2}}\sim 2_s^\ast-q.$$
\end{remark}
\begin{lemma}\label{lemma_4b}
Let $q\in(2,2_s^\ast)$,
    \begin{equation}\label{4f}
        Q(q):=\left(\frac{2_s^\ast-q}{2_s^\ast-2}\right)^{\frac{2_s^\ast-q}{q-2}}\quad\text{and}\quad G(q):=\frac{q-2}{2_s^\ast-2}Q(q) .
    \end{equation}
    Then $Q(q)\sim1$ and $G(q)\sim q-2$ and $\forall\lambda>0:$
    \begin{enumerate}
        \item $1<\tau(v_\lambda)\leq 1+G(q)\lambda^\sigma$,
        \item $m_{\lambda}>m_0\left(1-\lambda^\sigma \frac{N}{s}G(q)(1+G(q)\lambda^\sigma)^{\frac{N-2s}{2s}}\right).$
    \end{enumerate}
\end{lemma}
\begin{proof}\phantom{\qedhere}
    It's easy to show that $Q(q)\sim1$. Thus, $G(q)\sim (q-2)$.
    Since $v_\lambda\in\NN_\lambda$, using \eqref{tau_def}, we obtain
    \begin{equation*}
        \begin{aligned}\relax[v_\lambda]_{\H}^2&=\|v_\lambda\|_{2_s^\ast}^{2_s^\ast}+\lambda^\sigma (\|v_\lambda\|_q^q-\|v_\lambda\|_2^2).\\
        \tau(v_\lambda)&=1+\lambda^\sigma\frac{\|v_\lambda\|_q^q-\|v_\lambda\|_2^2}{\|v_\lambda\|_{2_s^\ast}^{2_s^\ast}}.
    \end{aligned}
    \end{equation*}
    Observe that $\frac{2(2_s^\ast-q)}{q(2_s^\ast-2)}<1$. Thus from \eqref{norm_2,q},
    $1<\tau(v_\lambda).$

    Using the interpolation inequality we can see that 
    \begin{equation}\label{4a}
    \frac{\|v_\lambda\|_q^q-\|v_\lambda\|_2^2}{\|v_\lambda\|_{2_s^\ast}^{2_s^\ast}}\leq \left(\frac{\|v_\lambda\|_2^2}{\|v_\lambda\|_{2_s^\ast}^{2_s^\ast}}\right)^{\theta_q}\left[1-\left(\frac{\|v_\lambda\|_2^2}{\|v_\lambda\|_{2_s^\ast}^{2_s^\ast}}\right)^{1-\theta_q}\right]\leq \theta_q^{\frac{\theta_q}{1-\theta_q}}(1-\theta_q)=G(q).
    \end{equation}
    where $\theta_q= \frac{2_s^\ast-q}{2_s^\ast-2}$. In the second inequality, we used the fact that if $\theta\in(0,1)$ $$\max_{x\geq0}x^\theta (1-x^{1-\theta})=\theta^{\frac{\theta}{1-\theta}}(1-\theta).$$
    This proves (1). Hence 
    \begin{equation}\label{tau_Gq_relation}
        \tau(v_\lambda)\leq 1+\lambda^\sigma G(q).
    \end{equation}
    Let $t_\lambda$ be the unique $t>0$ such that $(v_\lambda)_t\in\NN_0$. It's easy to check that $t_\lambda=(\tau(v_\lambda))^{\frac1{2s}}$. For the next part we observe that, by Lemma \ref{N_lambda_char},
    \begin{align}
        m_0=\sup_{t\geq0}J_0((U_\rho)_t)\leq \sup_{t\geq0}J_0((v_\lambda)_t)&=J_0((v_\lambda)_{t_\lambda})
        =J_\lambda((v_\lambda)_{t_\lambda})+\lambda^\sigma t_\lambda^N\left(\frac1q\|v_\lambda\|_q^q-\frac12\|v_\lambda\|_2^2\right)\nonumber\\
        &\leq m_\lambda+\lambda^\sigma (\tau(v_\lambda))^{\frac N{2s}}\left(\frac1q\|v_\lambda\|_q^q-\frac12\|v_\lambda\|_2^2\right).\label{4e}
    \end{align}
By \eqref{norm_2,q} and \eqref{4a},
\begin{equation*}
    \begin{aligned}
        \frac1q\|v_\lambda\|_q^q-\frac12\|v_\lambda\|_2^2\leq\frac12\left(\|v_\lambda\|_q^q-\|v_\lambda\|_2^2\right)\leq \|v_\lambda\|_q^q-\|v_\lambda\|_2^2\leq G(q) \|v_\lambda\|_{2_s^\ast}^{2_s^\ast}.
    \end{aligned}
\end{equation*}
By Corollary \ref{Pohozaev_corollary}, $m_\lambda=\frac sN [v_\lambda]_{\H}^2.$
Hence, from \eqref{4e}, it follows 
\begin{equation}
\begin{aligned}
m_0& \leq m_\lambda+\lambda^\sigma(\tau(v_\lambda))^\frac{N}{2s}G(q)\|v_\lambda\|_{2^*_s}^{2^*_s}
\leq m_\lambda+\lambda^\sigma(\tau(v_\lambda))^\frac{(N-2s)}{2s}G(q)[v_\lambda]^2_{\H}\\&\leq 
    m_\lambda\left[1+\lambda^\sigma \frac{N}{s}G(q)(1+G(q)\lambda^\sigma)^{\frac{N-2s}{2s}}\right],
\end{aligned}
\end{equation}
and 
\begin{equation*}
m_\lambda\geq \frac{m_0}{1+\lambda^\sigma \frac{N}{s}G(q)(1+G(q)\lambda^\sigma)^{\frac{N-2s}{2s}}}>m_0\left[1-\lambda^\sigma \frac{N}{s}G(q)(1+G(q)\lambda^\sigma)^{\frac{N-2s}{2s}}\right].\tag*{\qed}\end{equation*}
\end{proof}
Define $g_0(\rho):=\frac1q\|U_\rho\|_q^q-\frac12\|U_\rho\|_2^2$ for $\rho>0$. Then,
    \begin{equation*}g_0'(\rho)=0\iff \rho=\left(\frac{2(2_s^\ast-q)\|U_1\|_q^q}{q(2_s^\ast-2)\|U_1\|_2^2}\right)^{\frac{2_s^\ast-2}{2s(q-2)}}.\end{equation*}
    Since $g_0(\rho)$ is positive for small $\rho$ and negative for large $\rho$, there exists a unique $\rho_0>0$ given by $$\rho_0:=\left(\frac{2(2_s^\ast-q)\|U_1\|_q^q}{q(2_s^\ast-2)\|U_1\|_2^2}\right)^{\frac{2_s^\ast-2}{2s(q-2)}}$$ such that 
    \begin{equation}
    g_0(\rho_0)=\sup_{\rho>0}g_0(\rho)=\frac1q\left(\frac 2q\right)^{\frac{2_s^\ast-q}{q-2}}G(q)\left(\frac{\|U_1\|_q^{q(2_s^\ast-2)}}{\|U_1\|_2^{2(2_s^\ast-q)}}\right)^{\frac1{q-2}}.
    \end{equation}
Observe that 
\begin{equation*}
    \begin{aligned}
        \lim_{q\to2}\left(\frac{\|U_1\|_q^{q(2_s^\ast-2)}}{\|U_1\|_2^{2(2_s^\ast-q)}}\right)^{\frac1{q-2}}&=\exp\left(\lim_{q\to2}\frac{(2_s^\ast-2)\ln\left(\int_{\RR}|U_1|^q\,dx\right)-(2_s^\ast-q)\ln\left(\int_{\RR}|U_1|^2\,dx\right)}{q-2}\right)\\&=\exp\left((2_s^\ast-2)\lim_{q\to2}\frac{\ln\left(\int_{\RR}|U_1|^q\,dx\right)-\ln\left(\int_{\RR}|U_1|^2\,dx\right)}{q-2}\right)\int_{\RR}|U_1|^2\,dx.
    \end{aligned}
\end{equation*}
Using DCT, and the fact that $N>4s$, we can prove $\int_{\RR}|U_1|^q\,dx$ is a dif{}ferentiable function in $q$.
\begin{equation*}
\lim_{h\to0}\frac{\int_{\RR}|U_1|^{q+h}-|U_1|^h\,dx}{h}=\lim_{h\to0}\int_{\RR}|U_1|^q\frac{|U_1|^h-1}{h}\,dx=\int_{\RR}|U_1|^q\ln|U_1|\,dx.
\end{equation*}

Then, $\ln\left(\int_{\RR}|U_1|^q\,dx\right)$ is a dif{}ferentiable function of of $q$. So by chain rule,
\begin{equation*}
    \begin{aligned}
        \lim_{q\to2}\left(\frac{\|U_1\|_q^{q(2_s^\ast-2)}}{\|U_1\|_2^{2(2_s^\ast-q)}}\right)^{\frac1{q-2}}&=\exp\left((2_s^\ast-2)\frac{\int_{\RR}|U_1|^2\ln|U_1|\,dx}{\int_{\RR}|U_1|^2\,dx}\right)\int_{\RR}|U_1|^2\,dx.
    \end{aligned}
\end{equation*}
We conclude that $$c_0:=\inf_{q\in(2,2_s^\ast)}\left(\frac{\|U_1\|_q^{q(2_s^\ast-2)}}{\|U_1\|_2^{2(2_s^\ast-q)}}.\right)^{\frac1{q-2}}>0,$$
and 
\begin{equation}\label{g_0rho_0}
    g_0(\rho_0)\geq \frac{c_0}{q}\left(\frac 2q\right)^{\frac{2_s^\ast-q}{q-2}}G(q).
\end{equation}
\begin{lemma}\label{lemma_4c}
    Let $N>4s$. There exists $c_0>0$ independent of $q,\lambda$ such that for all small $\lambda>0,$
    $$m_\lambda<m_0-\lambda^\sigma\left(\frac{c_0}{q}\left(\frac2q\right)^{\frac{2_s^\ast-q}{q-2}}G(q)-\lambda^\sigma\frac{2Nm_0}{s(q-2)}G(q)^2\right).$$
\end{lemma}
\begin{proof}
Define \begin{equation}\label{U-0}
U_{0}(x):=U_{\rho_0}(x)={\rho_0}^{-\frac{N-2s}2}U_1(x/\rho_0).\end{equation} 
Let $\Tilde{t}_\lambda$ be the unique $t>0$ such that $tU_0\in\NN_\lambda$. As done in the previous lemma,
\begin{align}
    m_{\lambda}&\leq \sup_{t\geq0} J_{\lambda}(tU_0)= J_{\lambda}(\Tilde{t}_\lambda U_0)\nonumber\\&\leq m_0+\lambda^\sigma \left(\frac{\Tilde{t}_\lambda^2}2\|U_0\|_2^2-\frac{\Tilde{t}_\lambda^q}q\|U_0\|_q^q\right).\label{4d}
\end{align}

By Lemma \ref{N_lambda_char}, $\frac{d}{dt}J_\lambda(tU_0)\mid_{\Tilde{t}_\lambda}=0$. That implies,
$$\frac{Nm_0}s+\lambda^\sigma\|U_0\|_2^2= \Tilde{t}_\lambda^{2_s^\ast-2} \frac{Nm_0}s+\lambda^\sigma \Tilde{t}_\lambda^{q-2}\|U_0\|_q^q. $$
If $\Tilde{t}_\lambda\geq1$ then $\Tilde{t}_\lambda^{2_s^\ast-2} \geq \Tilde{t}_\lambda^{q-2}$ and $$\frac{Nm_0}s+\lambda^\sigma\|U_0\|_2^2> \Tilde{t}_\lambda^{q-2} \left(\frac{Nm_0}s+\lambda^\sigma \|U_0\|_q^q\right). $$
We recall that $g_0(\rho_0)=\frac1q\|U_0\|_q^q-\frac12\|U_0\|_2^2 >0$. Hence $\|U_0\|_q^q>\|U_0\|_2^2$. Therefore, $$\Tilde{t}_\lambda\leq \left(\frac{\frac{Nm_0}s+\lambda^\sigma \|U_0\|_2^2}{\frac{Nm_0}s+\lambda^\sigma \|U_0\|_q^q}\right)^{\frac1{q-2}}<1,$$
which is a contradiction. Therefore, $\Tilde{t}_\lambda<1$. Thus, 
$$\left(\frac{\frac{Nm_0}s+\lambda^\sigma \|U_0\|_2^2}{\frac{Nm_0}s+\lambda^\sigma \|U_0\|_q^q}\right)^{\frac1{q-2}}<\Tilde{t}_\lambda<1.$$
Let $$A_\lambda:=\frac{\|U_0\|_q^q-\|U_0\|_2^2}{\frac{Nm_0}s+\lambda^\sigma \|U_0\|_q^q}.$$
Then $A_\lambda<\frac s{Nm_0}\left(\|U_0\|_q^q-\|U_0\|_2^2\right)$ and because $\|U_0\|_q^q>\|U_0\|_2^2$, $$\left(\frac{\|U_0\|_2^2}{\|U_0\|_q^q}\right)^{\frac1{q-2}}<[1-\lambda^\sigma A_\lambda]^{\frac{1}{(q-2)}}<\Tilde{t}_\lambda<1.$$
Define $g(t):=\frac{t^2}2\|U_0\|_2^2-\frac{t^q}q\|U_0\|_q^q$ on $[0,1]$. Then $g$ has a unique maximum at $t_0=\left(\frac{\|U_0\|_2^2}{\|U_0\|_q^q}\right)^{\frac1{q-2}}$ and strictly decreasing in $(t_0,1)$.
We know that $[U_0]_{\H}^2=\|U_0\|_{2_s^\ast}^{2_s^\ast}$ and $m_0=\frac12[U_0]_{\H}^2-\frac1{2_s^\ast}\|U_0\|_{2_s^\ast}^{2_s^\ast}$. Thus, $[U_0]_{\H}^2=\|U_0\|_{2_s^\ast}^{2_s^\ast}=\frac{Nm_0}s.$ Hence  $g(\Tilde{t}_\lambda)<g([1-\lambda^\sigma A_\lambda]^{\frac{1}{(q-2)}}).$
Define for $t\in[0,1]$, $h(t):= g([1-t]^{\frac1{q-2}}).$ Then for small $t>0,$
$$h'(t)=\frac1{q-2}(1-t)^{\frac{4-q}{q-2}}\left((1-t)\|U_0\|_q^q-\|U_0\|_2^2\right)>0.$$
Thus for small $\lambda>0$, by the monotonicity of $g$ in $(t_0,1)$, $$g(\Tilde{t}_\lambda)<h(\lambda^\sigma A_\lambda)=\frac12\|U_0\|_2^2-\frac1q\|U_0\|_q^q+h'(\xi)\lambda^\sigma A_\lambda$$ for some $\xi\in(0,\lambda^\sigma A_\lambda).$ 
For small $\lambda>0,$ $$h'(\xi)\leq \frac{2}{q-2}(\|U_0\|_q^q-\|U_0\|_2^2).$$
Similar to \eqref{4a} we can see that 
$$\frac{\|U_0\|_q^q-\|U_0\|_2^2}{\|U_0\|_{2_s^\ast}^{2_s^\ast}}\leq G(q).$$
Hence, 
\begin{align*}
    g(\Tilde{t}_\lambda)&< \frac12\|U_0\|_2^2-\frac1q\|U_0\|_q^q+\frac{2s\lambda^\sigma}{Nm_0(q-2)}(\|U_0\|_q^q-\|U_0\|_2^2)^2\\
    &=-g(\rho_0)+\frac{2s\lambda^\sigma}{Nm_0(q-2)}\left(\frac{\|U_0\|_q^q-\|U_0\|_2^2}{\|U_0\|_{2_s^\ast}^{2_s^\ast}}\frac{Nm_0}s\right)^2\\
    &\leq-g(\rho_0)+\lambda^\sigma \frac{2Nm_0}{s(q-2)}(G(q))^2.
\end{align*}
Substituting this in \eqref{4d} and using \eqref{g_0rho_0}, we conclude the proof.
\end{proof}
\begin{corollary}\label{m_0m_lambda}
    Let $N>4s$. Then for small $\lambda>0$, $$m_0-m_\lambda\sim \lambda^\sigma.$$
\end{corollary}
\begin{proof}
Combining Lemma~\ref{lemma_4b} and Lemma~\ref{lemma_4c}, the proof follows. 
\end{proof}

\begin{lemma}\label{norm_q}
    Assume $N>4s$. Then for small $\lambda>0$,
    $$\frac{2q}{2_s^\ast-2}Q(q)m_0\geq \|v_{\lambda}\|_q^q\geq Q(q)\left(\frac{c_0}{q}\left(\frac{2_s^\ast}{q}\right)^{\frac{2_s^\ast-q}{q-2}}-\lambda^\sigma\frac{2Nm_0}{s}\frac{Q(q)}{2_s^\ast-2}\right)\frac{q}{(\tau(v_\lambda))^{\frac N{2s}}}.$$
    In particular, 
    \begin{equation}\label{v-lam-nonzero}
    \|v_\lambda\|_2^2\sim 2^\ast-q,\quad \|v_\lambda\|_q^q\sim 1.
    \end{equation}
\end{lemma}
\begin{proof}
    By $v_\lambda\in\NN_\lambda$, Corollary \ref{Pohozaev_corollary} and \eqref{norm_2,q}
    $$m_\lambda=\frac{s}{N}[v_\lambda]_{\H}^2=\frac sN\|v_\lambda\|_{2_s^\ast}^{2_s^\ast}+\lambda^\sigma\frac{q-2}{2q}\|v_\lambda\|_q^q.$$
Thus, $$\lambda^\sigma\frac{q-2}{2q}\|v_\lambda\|_q^q=\frac{\tau(v_\lambda)-1}{\tau(v_\lambda)}m_\lambda.$$
Therefore, using Lemma \ref{lemma_4b} and $m_\lambda<m_0,$
$$\lambda^\sigma\frac{q-2}{2q}\|v_\lambda\|_q^q<\lambda^\sigma G(q)m_0.$$
Substituting $\frac{G(q)}{q-2} =\frac{Q(q)}{2^*_s-2}$ from \eqref{4f} into the above relation, we obtain the LHS estimate of this lemma.
Using \eqref{norm_2,q} in \eqref{4e}, we obtain,
$$m_0\leq m_\lambda+\lambda^\sigma(\tau(v_\lambda))^{\frac{N}{2s}}\frac{q-2}{q(2_s^\ast-2)}\|v_\lambda\|_q^q.$$
This led us 
\begin{align*}
    \|v_\lambda\|_q^q&\geq \frac{(m_0-m_\lambda)q(2_s^\ast-2)}{\lambda^\sigma (\tau(v_\lambda))^{\frac{N}{2s}}(q-2)}\\&>\frac{q(2_s^\ast-2)}{(\tau(v_\lambda))^{\frac{N}{2s}}} \left(\frac{c_0}q\left(\frac2q\right)^{\frac{2_s^\ast-q}{q-2}}\frac{G(q)}{q-2}-\lambda^\sigma\frac{2Nm_0}{s}\left(\frac{G(q)}{q-2}\right)^2\right).\quad(\text{by Lemma }\ref{lemma_4c})
\end{align*}
Substituting $\frac{G(q)}{q-2} =\frac{Q(q)}{2^*_s-2}$ from \eqref{4f} into the above relation, we obtain the RHS estimate of this lemma.
\end{proof}
By Lemma \ref{rescaling_lemma} and the above lemma, $$\|u_\lambda\|_2^2\sim(2_s^\ast-q)\lambda^\sigma, \quad\|u_\lambda\|_q^q\sim\lambda^{\sigma-1}=\lambda^\frac{2^*_s-q}{q-2}.$$
Using Corollaries \ref{m_0m_lambda}, \ref{Pohozaev_corollary} and Lemma \ref{rescaling_lemma}, we get $[u_\lambda]_{\H}^2\sim1.$ 
Again since, $[v_\lambda]_{\H}^2=\|v_\lambda\|_{2_s^\ast}^{2_s^\ast}+c\lambda^\sigma\|v_\lambda\|_q^q$, $\|u_\lambda\|_{2_s^\ast}^{2_s^\ast}\sim1.$ Hence, \eqref{asymp-ii} follows.  Once we prove $v_\lambda\to U_0$ in $C^{2s-\delta}(\RR)$ it's obvious that $u_\lambda(0)\sim\lambda^{\frac{1}{q-2}}.$ Using Corollaries \ref{Pohozaev_corollary} and \ref{m_0m_lambda}, we get the convergence rate \eqref{conv_rate}.

\begin{proof}[\bf {Proof of Theorem \ref{Asymp_1} for $N>4s$}:] From the discussion following Lemma \ref{v_lambda_bdd}, there exists $v_0\in H_r^s(\RR)$ satisfying the critical Emden-Fowler equation $\fra v=v^{2_s^\ast-1}$ such that,
\begin{equation}
    v_\lambda\rightharpoonup v_0\quad\text{in }\H,\quad v_\lambda\to v_0\quad\text{in } L^p(\RR)\quad \forall\, p\in(2,2_s^\ast)
\end{equation}
and 
\begin{equation}
    v_\lambda(x)\to v_0(x)\quad \text{a.e. }\RR,\quad v_\lambda\to v_0\quad\text{in } L_{\text{loc}}^2(\RR).
\end{equation}
\noindent\textbf{Step-1: $\dot{H}^s$ convergence.} Now using Lemma \ref{norm_q} and Corollary \ref{m_0m_lambda}, $$J_0(v_\lambda)=J_\lambda(v_\lambda)+\lambda^\sigma\left(\frac1q\|v_\lambda\|_q^q-\frac12\|v_\lambda\|_2^2\right)=m_\lambda+o(1)=m_0+o(1)$$
and 
$$J_0'(v_\lambda)v=J_\lambda'(v_\lambda)v+\lambda^\sigma \int_{\RR}|v_\lambda|^{q-2}v_\lambda v\,dx-\lambda^\sigma \int_{\RR}v_\lambda v\,dx=o(1).$$
Therefore, $\{v_\lambda\}$ is a PS sequence for $J_0$ at the level $m_0$. Further, by \eqref{v-lam-nonzero} weak limit of $v_\lambda$ has to be non-zero. Hence, by Lemma \ref{PS_theorem_2}  
\begin{align*}
    v_{\lambda}&=U_\rho+\sum_{j=1}^k (R_{\lambda}^{(j)})^{-\frac{N-2s}2}v^{(j)}\left(\frac{x-x_{\lambda}^{(j)}}{R_{\lambda}^{(j)}}\right)+o(1)\quad\text{in }\HHH.\\
    m_0&=J_0(U_\rho)+\sum_{j=1}^kJ_0(v^{(j)}),
\end{align*}
where $U_\rho$ is a unique positive solution of $\fra v=v^{2_s^\ast-1}$ of the form \eqref{Urho}. Since $v^{(j)}$ are non trivial solutions of the limit equation $\fra v=v^{2_s^\ast-1}$, $J_0(v^{(j)})\geq m_0$ and also notice that $J_0(U_\rho)\geq0.$
By Lemma \ref{norm_q}, $U_\rho\neq0$. Thus, $k=0$ and $v_\lambda\to U_\rho$ in $\HHH$ for some $\rho>0$.

\noindent\textbf{Step-2: $C^{2s-\delta}$ convergence.} Notice that for $\lambda\leq1$, $v_\lambda$ is a subsolution of $$\fra v=v^{2_s^\ast-1}+v^{q-1}.$$ Since the coefficients of the above equation are independent of $\lambda$, following the arguments of \cite[Proposition~5.1.1]{DMV}, we can argue that any positive subsolution $v$ of the above equation satisfies $$\sup_{\RR}v(x)\leq C\|v\|_{2_s^\ast}^{\gamma}$$ for some $\gamma>0$ and $C>0$ independent of $\lambda$. Thus $\{v_\lambda\}_{\lambda\leq1}$ is uniformly bounded in $L^{\infty}(\RR)$ w.r.t. $\lambda\in (0,1]$. Let $r>2$ and choose $2<t<\min(r,2_s^\ast)$. Then as $\lambda\to0$, $$\int_{\RR}|v_\lambda-U_\rho|^r\,dx\leq \|v_\lambda-U_\rho\|_{\infty}^{r-t}\int_{\RR}|v_\lambda-U_\rho|^t\,dx\leq C\int_{\RR}|v_\lambda-U_\rho|^t\,dx\to0,$$
where for the last limit, we have used Lemma~\ref{Strauss_compactness_lemma}. Thus $v_\lambda\to U_\rho$ in $L^p(\RR)$ for all $p\in(2,\infty)$. Therefore, in particular, $v_\lambda$ is uniformly bounded in $L^p$ for $p\in(2,\infty)$. 

Now choosing $r>\frac N{2s}$ large enough such that $r(q-1),r>2$, for $\lambda\leq1$, we can obtain
\begin{equation*}
    \|\fra v_\lambda-\fra U_\rho\|_r\leq \|v_\lambda^{2_s^\ast-1}-U_\rho^{2_s^\ast-1}\|_r+\lambda^\sigma\|v_\lambda^{q-1}-v_\lambda\|_r\leq  \|v_\lambda^{2_s^\ast-1}-U_\rho^{2_s^\ast-1}\|_r+ C\lambda^\sigma,
\end{equation*}
where in the last inequality we used $r(q-1),\, r>2$ to get $\|v_\lambda^{q-1}-v_\lambda\|_r\leq C$. Using the elementary inequality: for every $0<p<\infty$ and $\epsilon>0$, there exists $C(\epsilon,p)>0$ such that $\big||a+b|^p-|a|^p\big|\leq \epsilon|a|^p+C(\epsilon)|b|^p $, we obtain
\begin{align*}
    \int_{\RR}|v_\lambda^{2_s^\ast-1}-U_\rho^{2_s^\ast-1}|^r\,dx&\leq \int_{\RR}\left|\epsilon|U_\rho|^{2_s^\ast-1}+C_\epsilon|v_\lambda-U_\rho|^{2_s^\ast-1}\right|^r\,dx\\
    &\leq 2^{r-1}\epsilon^r\int_{\RR}|U_\rho|^{r(2_s^\ast-1)}+C_\epsilon|v_\lambda-U_\rho|^{r(2_s^\ast-1)}\,dx,
\end{align*}
for any $\epsilon>0$.
Using the definition of $U_\rho$ from \eqref{Urho}, it follows that $U_\rho\in L^{r(2_s^\ast-1)}$.
Therefore, as $v_\lambda\to U_\rho$ in $L^{r(2_s^\ast-1)}(\RR)$, we obtain $\|\fra v_\lambda-\fra U_\rho\|_r\to0$ as $\lambda\to0$. Now \cite[Theorem~1.6(iii)]{RS3} yields
$$\|v_\lambda-U_\rho\|_{C^{\alpha}(\RR)}\leq \|\fra v_\lambda-\fra U_\rho\|_r, \quad\alpha=2s-\frac Nr$$
and thus $v_\lambda\to U_\rho$ in $C^{2s-\delta}(\RR)$ where $\delta=\frac Nr$. Thus, \eqref{asymp-i} follows. 

\noindent\textbf{Step-3: $\|v_\lambda\|_2^2\to\|U_0\|_2^2$ as $\lambda\to0$,} where $U_0:=U_{\rho_0}$ as defined in \eqref{U-0}. 

 For this we first prove that $v_0=U_0$. Define $M_\lambda:=\frac{v_\lambda(0)}{U_0(0)}$ and 
\begin{equation}\label{rescale_2}
w_\lambda(x):=M_\lambda^{-1}v_\lambda\left(M_\lambda^{-\frac{2}{N-2s}}x\right).
\end{equation}
Therefore, $w_\lambda$ solves 
\begin{equation}\label{w_lambda}
\begin{cases}
    \fra w_\lambda+\lambda^\sigma M_\lambda^{2-2_s^\ast}w_\lambda&=w_\lambda^{2_s^\ast-1}+\lambda^\sigma M_\lambda^{q-2_s^\ast}w_\lambda\\
    \hfill w_\lambda(0)&=U_0(0).
\end{cases}
\end{equation}
Moreover,
$$\|w_\lambda\|_2^2=M_\lambda^\frac{4s}{N-2s}\|v_\lambda\|_2^2, \quad [w_\lambda]_{\H}^2= [v_\lambda]_{\H}^2. $$
Since $v_\lambda(0)\to U_\rho(0)\neq0$ as $\lambda\to0$, we get $M_\lambda\sim 1$. Therefore,
$$[w_\lambda]_{\H}^2\sim  \|w_\lambda\|_{2^*_s}^{2^*_s}\sim\|w_\lambda\|_q^q \sim1, \quad \|w_\lambda\|_2^2\sim(2_s^\ast-q).$$
Thus up to a subsequence $w_\lambda\rightharpoonup \tilde w$ in $\H$ for some $\tilde w\in\H$. Hence passing the limit $\lambda\to 0$ in the weak formulation of \eqref{w_lambda} yields that $\tilde w$ is a positive weak solution of \eqref{Emden-Fowler}. Moreover $w_\lambda(0)=U_0(0)$ for all $\lambda>0$ implies $\tilde w(0)=U_0(0)$ and hence $\tilde w=U_0$ i.e., $w_\lambda \rightharpoonup U_{\rho_0}$. Now repeating the arguments we did for $v_\lambda$, we see that $w_\lambda\to U_{\rho_0}$ in $\dot{H}^s(\RR)\cap C^{2s-\delta}(\RR)$. Moreover, 
 \eqref{rescale_2} implies $[w_\lambda]_{\H}^2=[v_\lambda]_{\H}^2$. Combining this with the fact that $w_\lambda\to U_{\rho_0}$ and $v_\lambda\to U_{\rho}$ in $\dot{H}(\RR)\cap C^{0,\alpha}(\RR)$ for some $\alpha>0$, we conclude that 
 $U_{\rho_0}=U_{\rho}$. 
 
Observe that by direct computation using the definition of $U_{\rho_0}$, it follows that $$\|U_{\rho_0}\|_2^2=\frac{2(2_s^\ast-q)}{q(2_s^\ast-2)}\|U_{\rho_0}\|_q^q.$$
 Thus using \eqref{norm_2,q}, we see that
\begin{equation*}
    \|v_\lambda\|_2^2-\|U_{\rho_0}\|_2^2=\frac{2(2_s^\ast-q)}{q(2_s^\ast-2)}(\|v_\lambda\|_q^q-\|U_{\rho_0}\|_q^q)\to 0 \quad\text{as}\quad \lambda\to 0,
\end{equation*}
where for the last limit we have used the radial lemma $v_\lambda\to U_{\rho_0}$ in $L^q(\RR)$.

\noindent\textbf{Step-4: $\H$ convergence} From step 1 and step 3, we got $\|v_\lambda\|_{\H}\to\|U_{\rho_0}\|_{\H}$ as $\lambda\to0$.  Since Hilbert spaces are uniformly convex, using $v_\lambda\rightharpoonup U_{\rho_0}$ in $\H$ we get $v_\lambda\to U_{\rho_0}$ in $\H$. Hence, $v_\lambda\to U_{\rho_0}$ in $\H\cap C^{2s-\delta}(\RR)$ where $\delta\in(0,2s)$.
\end{proof}
For the remaining case {\bf $2s<N\leq 4s$}, we first present a lemma analogous to Lemma~\ref{lemma_4c}.
\begin{lemma}\label{l:low-dimension}
There exists a constant $\varpi=\varpi(q)>0$ such that for all small $\lambda>0$,
\begin{equation}\label{4h0}
    m_\lambda\leq\begin{cases}
    m_0-\lambda^\sigma\left(\ln\frac1{\lambda}\right)^{-\frac{4-q}{q-2}}\varpi, \text{ if }N=4s,\\
    m_0-\lambda^{\frac{2(N-2s)}{(N-2s)q-4s}}\varpi, \text{ if }2s<N<4s, \, \, q>\frac{4s}{N-2s}.
\end{cases}
\end{equation}
\end{lemma}
\begin{proof}
    Let $R\gg1$ be a large parameter and $\eta_R\in C_0^{\infty}(\RR)$ is a cutoff function such that $\eta_R\equiv1$ on $B(0,R)$, $\eta_R\equiv0$ on $\RR\setminus B(0,2R)$, $0\leq\eta_R\leq1$ when $R\leq|x|\leq2R$ and $|\nabla\eta_R(x)|\leq\frac2R$.

    \noindent\textbf{Claim:} For $\ell\gg1$, 
    \begin{align}
        [\eta_\ell U_1]_{\H}^2&\leq[U_1]_{\H}^2+O(\ell^{2s-N})\label{4h1}.\\
        \|\eta_\ell U_1\|_{2_s^\ast}^{2_s^\ast}&=\|U_1\|_{2_s^\ast}^{2_s^\ast}+O(\ell^{-N})\label{4h2}.\\
        \|\eta_\ell 
 U_1\|_{q}^{q}&=\|U_1\|_{q}^{q}+o(1)\label{4h3}.\\
        \|\eta_\ell U_1\|_2^2&\sim\begin{cases}
            \ln \ell (1+o(1)),\text{ if } N=4s,\\
            \ell^{4s-N}(1+o(1)),\text{ if }2s<N<4s.
        \end{cases}\label{4h4}
    \end{align}
      It is easy to see that \eqref{4h1}, \eqref{4h2} and one side of \eqref{4h4} follow from Proposition 21 and Proposition 22 of \cite{SV} by a change of variables. For the remaining parts we proceed like \cite[Section~6.1]{BM2}. 
    \begin{align*}
        \|\eta_\ell U_1\|_q^q&\leq \int_{B_{2\ell}}|U_1|^q\,dx=\|U_1\|_q^q-\int_{\RR\setminus B_{2\ell}}|U_1|^q\,dx\\&\cong\|U_1\|_q^q-C\int_{\RR\setminus B_{2\ell}}\frac1{|x|^{q(N-2s)}}\,dx\\&=\|U_1\|_q^q-C\ell^{N-q(N-2s)},
    \end{align*}
    since the given hypothesis $2s<N\leq 4s$ and $q>\frac{4s}{N-2s}$ implies $q>\frac{N}{N-2s}$.
    Similarly, 
    \begin{equation*}
        \|\eta_\ell U_1\|_q^q\geq \int_{B_{\ell}}|U_1|^q\,dx=\|U_1\|_q^q-\int_{\RR\setminus B_{\ell}}|U_1|^q\,dx\cong\|U_1\|_q^q-C\ell^{N-q(N-2s)},
    \end{equation*}
which yields \eqref{4h3}.
Let $\ell_0\gg1$ be fixed then for $\ell>\ell_0$,
\begin{align*}
    \|\eta_\ell U_1\|_2^2&\leq \int_{B_{2\ell}}|U_1|^2\,dx=\int_{B_{\ell_0}}|U_1|^2\,dx+\int_{B_{2\ell}\setminus B_{\ell_0}}|U_1|^2\,dx =\int_{B_{\ell_0}}|U_1|^2\,dx+C\int_{\ell_0}^{2\ell}r^{4s-N-1}\,dr.
\end{align*}
Thus we get \eqref{4h4}. Using Lemma~\ref{N_lambda_char}, for $\rho>0$ to be chosen later, we have 
\begin{align*}
    m_\lambda\leq \sup_{t\geq0} J_\lambda((\eta_R U_{\rho})_t)=J_\lambda((\eta_RU_\rho)_{a_\lambda})&\leq \sup_{t\geq0}\left(\frac{t^{N-2s}}{2}[\eta_R U_{\rho}]_{\H}^2-\frac{t^N}{2_s^\ast}\|\eta_R U_{\rho}\|_{2_s^\ast}^{2_s^\ast}\right)\\&\quad-\lambda^\sigma a_\lambda^N\left[\frac1q\|\eta_R U_{\rho}\|_q^q-\frac12\|\eta_R U_{\rho}\|_2^2\right]\\&=A-\lambda^\sigma B,
\end{align*}
where 
\begin{equation}
    a_\lambda=\left(\frac{\frac1{2_s^\ast}[\eta_R U_{\rho}]_{\H}^2}{\frac1{2_s^\ast}\|\eta_R U_{\rho}\|_{2_s^\ast}^{2_s^\ast}+\lambda^\sigma\left[\frac1q\|\eta_R U_{\rho}\|_q^q-\frac12\|\eta_R U_{\rho}\|_2^2\right]}\right)^{\frac1{2s}}.
\end{equation}
Set $\ell=\frac R\rho$ then using the fact $[U_1]_{\H}^2=\|U_1\|_{2_s^\ast}^{2_s^\ast}=\frac{Nm_0}{s}$,
\begin{align}
    A&=\frac sN\frac{[\eta_R U_{\rho}]_{\H}^{\frac Ns}}{\|\eta_R U_{\rho}\|_{2_s^\ast}^{\frac Ns}}=\frac sN\frac{[\eta_\ell U_1]_{\H}^{\frac Ns}}{\|\eta_\ell U_1\|_{2_s^\ast}^{\frac Ns}}\leq\frac sN\frac{\left[\frac{Nm_0}{s}+O(\ell^{2s-N})\right]^{\frac{N}{2s}}}{\left[\frac{Nm_0}{s}+O(\ell^{-N})\right]^{\frac{N}{2_s^\ast s}}}\nonumber\\
    &=m_0\frac{\left[1+O(\ell^{2s-N})\right]^{\frac{N}{2s}}}{\left[1+O(\ell^{-N})\right]^{\frac{N}{2_s^\ast s}}}\leq m_0(1+O(\ell^{2s-N}))(1-O(\ell^{-N})+O(\ell^{-2N})))\nonumber\\&\leq m_0+O(\ell^{2s-N}).\label{4i}
\end{align}
Define $\phi:(0,\infty)\to\R$ by $$\phi(\rho):=\frac1q\|\eta_R U_{\rho}\|_q^q-\frac12\|\eta_R U_{\rho}\|_2^2=\frac1q\rho^{N-\frac{q(N-2s)}{2}}\|\eta_\ell U_1\|_q^q-\frac12\rho^{2s}\|\eta_\ell U_1\|_2^2.$$
Then it can be checked that $$\phi'(\rho)=0\implies \rho=\left(\frac{2(2_s^\ast-q)}{q(2_s^\ast-2)}\,\frac{\|\eta_\ell U_1\|_q^q}{\|\eta_\ell U_1\|_2^2}\right)^{\frac{2}{(N-2s)(q-2)}}.$$
Denote this $\rho$ by $\rho_1$. Therefore, 
$$\phi(\rho_1)=\max_{\rho>0} \phi(\rho)=\frac1q\left(\frac2q\right)^{\frac{2_s^\ast-q}{q-2}}G(q)\left(\frac{\|\eta_\ell U_1\|_q^{q(2_s^\ast-2)}}{\|\eta_\ell U_1\|_2^{2(2_s^\ast-q)}}\right)^{\frac{1}{q-2}}.$$
Using the interpolation inequality $\|\eta_\ell U_1\|_q^q\leq \|\eta_\ell U_1\|_2^{\frac{2(2_s^\ast-q)}{2_s^\ast-2}}\|\eta_\ell U_1\|_{2_s^\ast}^{\frac{2_s^\ast(q-2)}{2_s^\ast-2}}$, \\
 we have 
$
\phi(\rho_1)\leq \frac1q\left(\frac2q\right)^{\frac{2_s^\ast-q}{q-2}}G(q)\|U_1\|_{2_s^\ast}^{2_s^\ast}.
$
From now on we let $\rho=\rho_1$. Thus we get 
\begin{align}
    m_\lambda&= \frac sN\frac{[\eta_\ell U_1]_{\H}^{\frac Ns}}{\|\eta_\ell U_1\|_{2_s^\ast}^{\frac Ns}}-\lambda^\sigma \left(\frac{[\eta_\ell U_1]_{\H}^2}{\|\eta_\ell U_1\|_{2_s^\ast}^{2_s^\ast}+\lambda^\sigma2_s^\ast\phi(\rho_1)}\right)^{\frac{N}{2s}}\phi(\rho_1)\nonumber\\
    &\leq \frac sN\frac{[\eta_\ell U_1]_{\H}^{\frac Ns}}{\|\eta_\ell U_1\|_{2_s^\ast}^{\frac Ns}}\left[1-\frac{N\lambda^\sigma}{s}\frac{\phi(\rho_1)}{\|\eta_\ell U_1\|_{2_s^\ast}^{2_s^\ast}}\left(\frac{1}{1+\lambda^\sigma \frac{2_s^\ast\phi(\rho_1)}{\|\eta_\ell U_1\|_{2_s^\ast}^{2_s^\ast}}}\right)^{\frac{N}{2s}}\right]\nonumber\\
    &\leq \frac sN\frac{[\eta_\ell U_1]_{\H}^{\frac Ns}}{\|\eta_\ell U_1\|_{2_s^\ast}^{\frac Ns}}\left[1-\frac{N\lambda^\sigma}{s}\frac{\phi(\rho_1)}{\|\eta_\ell U_1\|_{2_s^\ast}^{2_s^\ast}}\left(1-\lambda^\sigma\frac{N2_s^\ast}{2s}\frac{\phi(\rho_1)}{\|\eta_\ell U_1\|_{2_s^\ast}^{2_s^\ast}}\right)\right]\nonumber\\
    &\leq \frac sN\frac{[\eta_\ell U_1]_{\H}^{\frac Ns}}{\|\eta_\ell U_1\|_{2_s^\ast}^{\frac Ns}}\left[1-\frac{N\lambda^\sigma}{2s}\frac{\phi(\rho_1)}{\|\eta_\ell U_1\|_{2_s^\ast}^{2_s^\ast}}\right]\quad(\text{for }\lambda\text{ small})\nonumber\\
    &\leq \frac sN\frac{[\eta_\ell U_1]_{\H}^{\frac Ns}}{\|\eta_\ell U_1\|_{2_s^\ast}^{\frac Ns}}\left[1-\lambda^\sigma\frac{\phi(\rho_1)}{2m_0}\right] 
    \leq A\left[1-\lambda^\sigma\frac{\phi(\rho_1)}{2m_0}\right].
    \label{4j}
\end{align}
For the rest of the proof, we consider the cases $N=4s$ and $2s<N<4s$ separately.\\
\textbf{Case 1: $N=4s$}\\
Using \eqref{4h3} and \eqref{4h4}, we have for $\ell\gg1$ \begin{align}
    \phi(\rho_1)&=\frac1q\left(\frac2q\right)^{\frac{2_s^\ast-q}{q-2}}G(q)\left(\frac{\|\eta_\ell U_1\|_q^{q(2_s^\ast-2)}}{\|\eta_\ell U_1\|_2^{2(2_s^\ast-q)}}\right)^{\frac{1}{q-2}}\nonumber\\
    &=\frac1q\left(\frac2q\right)^{\frac{2_s^\ast-q}{q-2}}G(q)\left(\frac{\| U_1\|_q^{q(2_s^\ast-2)}+o(1)}{[C\ln\ell(1+o(1))]^{(2_s^\ast-q)}}\right)^{\frac{1}{q-2}}\nonumber\\
    &=(\ln\ell)^{-\frac{2_s^\ast-q}{q-2}}\frac1q\left(\frac2{C q}\right)^{\frac{2_s^\ast-q}{q-2}}G(q)\left(\| U_1\|_q^{q(2_s^\ast-2)}+o(1)\right)^{\frac{1}{q-2}}\nonumber\\
    &\geq (\ln\ell)^{-\frac{2_s^\ast-q}{q-2}}\frac1{2q}\left(\frac2{C q}\right)^{\frac{2_s^\ast-q}{q-2}}G(q)\| U_1\|_q^{\frac{q(2_s^\ast-2)}{q-2}}.\label{4k1}
\end{align}
Now using \eqref{4i}, \eqref{4k1} in \eqref{4j} we get 
$$m_\lambda\leq (m_0+O(\ell^{2s-N}))\left[1-\lambda^\sigma (\ln\ell)^{-\frac{2_s^\ast-q}{q-2}}\frac1{4qm_0}\left(\frac2{C q}\right)^{\frac{2_s^\ast-q}{q-2}}G(q)\| U_1\|_q^{\frac{q(2_s^\ast-2)}{q-2}}\right].$$
Let $\ell=\frac1{\lambda^M}$ where $M$ is to be chosen later. Then 
$$m_\lambda\leq (m_0+O(\lambda^{2Ms})\left[1-\lambda^\sigma (M\ln\frac1{\lambda})^{-\frac{2_s^\ast-q}{q-2}}\frac1{4qm_0}\left(\frac2{C q}\right)^{\frac{2_s^\ast-q}{q-2}}G(q)\| U_1\|_q^{\frac{q(2_s^\ast-2)}{q-2}}\right].$$
If $M>\frac1{(q-2)s}$ then $2Ms>\frac{2}{q-2}=\frac{2_s^\ast-2}{q-2}=\sigma$. Thus
\begin{equation}
    m_\lambda\leq m_0-\lambda^\sigma(\ln\frac1{\lambda})^{-\frac{2_s^\ast-q}{q-2}}\frac1{4q}\left(\frac2{C qM}\right)^{\frac{2_s^\ast-q}{q-2}}G(q)\| U_1\|_q^{\frac{q(2_s^\ast-2)}{q-2}}.
\end{equation}
Letting $\varpi=\frac1{4q}\left(\frac2{C qM}\right)^{\frac{2_s^\ast-q}{q-2}}G(q)\| U_1\|_q^{\frac{q(2_s^\ast-2)}{q-2}}$, we obtain \eqref{4h0} in the case $N=4s$.

\noindent\textbf{Case 2: $2s<N<4s$ and $\frac{4s}{N-2s}<q<2_s^\ast$}\\
As before using \eqref{4h3} and \eqref{4h4}, we get
\begin{align}
    \phi(\rho_1)&=\frac1q\left(\frac2q\right)^{\frac{2_s^\ast-q}{q-2}}G(q)\left(\frac{\| U_1\|_q^{q(2_s^\ast-2)}+o(1)}{[C\ell^{4s-N}(1+o(1))]^{(2_s^\ast-q)}}\right)^{\frac{1}{q-2}}\nonumber\\
    &=\ell^{-\frac{(2_s^\ast-q)(4s-N)}{q-2}}\frac1q\left(\frac2{C q}\right)^{\frac{2_s^\ast-q}{q-2}}G(q)\left(\| U_1\|_q^{q(2_s^\ast-2)}+o(1)\right)^{\frac{1}{q-2}}\nonumber\\
    &\geq \ell^{-\frac{(2_s^\ast-q)(4s-N)}{q-2}}\frac1{2q}\left(\frac2{C q}\right)^{\frac{2_s^\ast-q}{q-2}}G(q)\| U_1\|_q^{\frac{q(2_s^\ast-2)}{q-2}}\label{4k2}.
\end{align}
Using \eqref{4k2}, \eqref{4i} in \eqref{4j}, we get 
$$m_\lambda\leq (m_0+O(\ell^{2s-N}))\left[1-\lambda^\sigma \ell^{-\frac{(2_s^\ast-q)(4s-N)}{q-2}}\frac1{4qm_0}\left(\frac2{C q}\right)^{\frac{2_s^\ast-q}{q-2}}G(q)\| U_1\|_q^{\frac{q(2_s^\ast-2)}{q-2}}\right].$$

Let $\ell=\delta^{-1}\lambda^{\frac{-2}{(N-2s)q-4s}}$ where $\delta$ is to be chosen later. Then 

$$m_\lambda\leq (m_0+\delta^{N-2s}O(\lambda^{\frac{2(N-2s)}{(N-2s)q-4s}}))\left[1-\lambda^{\frac{2(N-2s)}{(N-2s)q-4s}} \delta^{\frac{(2_s^\ast-q)(4s-N)}{q-2}}\frac1{4qm_0}\left(\frac2{C q}\right)^{\frac{2_s^\ast-q}{q-2}}G(q)\| U_1\|_q^{\frac{q(2_s^\ast-2)}{q-2}}\right].$$
Since $(N-2s)q>4s$, we have $\frac{(2_s^\ast-q)(4s-N)}{q-2}<N-2s$. Thus we can choose a small $\delta>0$ such that 
$$m_\lambda\leq m_0-\lambda^{\frac{2(N-2s)}{(N-2s)q-4s}} \frac1{4q}\left(\frac{2\delta^{4s-N}}{C q}\right)^{\frac{2_s^\ast-q}{q-2}}G(q)\| U_1\|_q^{\frac{q(2_s^\ast-2)}{q-2}}.$$

Letting $\varpi=\frac1{4q}\left(\frac{2\delta^{4s-N}}{C q}\right)^{\frac{2_s^\ast-q}{q-2}}G(q)\| U_1\|_q^{\frac{q(2_s^\ast-2)}{q-2}}$, we get \eqref{4h0}.
\end{proof}
\begin{corollary}\label{m_0m_lambda1}
For small $\lambda>0$
    $$\lambda^\sigma\gtrsim  m_0-m_\lambda\gtrsim \begin{cases}
        \lambda^\sigma\left(\ln\frac1{\lambda}\right)^{-\frac{4-q}{q-2}},\text{ if } $N = 4s,$\\
        \lambda^{\frac{2(N-2s)}{(N-2s)q-4s}},\text{ if } 2s<N<4s,\,\,  q>\frac{4s}{N-2s}.
    \end{cases}$$
\end{corollary}
\begin{proof}
Proof follows from Lemma~\ref{lemma_4b} and Lemma~\ref{l:low-dimension}.
\end{proof}

{\bf Proof of Theorem \ref{Asymp_1} for $2s<N\leq 4s$:}

By Lemma \ref{v_lambda_bdd}, $v_\lambda$ is bounded in $\H$. Thus 
from the discussion following Lemma \ref{v_lambda_bdd}, there exists $v_0\in H_r^s(\RR)$ such that $v_\lambda\rightharpoonup v_0$ in $\H$  and satisfies the critical Emden-Fowler equation $\fra v=v^{2_s^\ast-1}$. Moreover, $v_0\geq 0$ as $v_\lambda>0$. By Fatou's lemma 
$\|v_0\|_2^2\leq\lim\inf_{\lambda\to 0}\|v_\lambda\|_2^2<\infty$. On the other hand, for $N\leq4s$ Aubin-Talenti functions $U_\rho\not\in L^2(\RR)$. Hence $v_0=0$. 

Further, as done in the proof of Theorem \ref{Asymp_1} with $N>4s$, by Corollary \ref{m_0m_lambda1}, we have 
$$J_0(v_\lambda)=J_\lambda(v_\lambda)+\lambda^\sigma\left(\frac1q\|v_\lambda\|_q^q-\frac12\|v_\lambda\|_2^2\right)=m_\lambda+o(1)=m_0+o(1)$$
and 
$$J_0'(v_\lambda)v=J_\lambda'(v_\lambda)v+\lambda^\sigma \int_{\RR}|v_\lambda|^{q-2}v_\lambda v\,dx-\lambda^\sigma \int_{\RR}v_\lambda v\,dx=o(1).$$
Therefore $\{v_\lambda\}$ is a PS sequence for $J_0$ at the level $m_0$.  Hence by Lemma \ref{PS_theorem_2}  
\begin{align*}
    v_{\lambda}&=\sum_{j=1}^k (R_{\lambda}^{(j)})^{-\frac{N-2s}2}v^{(j)}\left(\frac{x}{R_{\lambda}^{(j)}}\right)+o(1)\quad\text{in }\HHH\\
    m_0&=\sum_{j=1}^kJ_0(v^{(j)}).
\end{align*}
Since $v^{(j)}$ are non trivial solutions of the limit equation $\fra v=v^{2_s^\ast-1}$, $J_0(v^{(j)})\geq m_0$ as a result $k=1$ and $v^{(1)}=U_\rho$ for some $\rho\in(0,\infty)$. Therefore we conclude that 
\begin{equation}\label{17-8-1}
v_\lambda-\xi_\lambda^{-\frac{N-2s}{2}}U_1(\xi_\lambda^{-1}\cdot)\to0 \qquad\text{in}\quad \HHH
\end{equation}
as $\lambda\to0$, where $\xi_\lambda:=\rho R_\lambda^{(1)}\to0$ as $\lambda\to0$. From \eqref{17-8-1} it follows that $w_\lambda(x):=\xi_\lambda^{\frac{N-2s}{2}}v_\lambda(\xi_\lambda x)\to U_1$ in $\HHH$. Thus $w_\lambda\to U_1$ in $L^q_{\text{loc}}(\RR)$. Further, 
$$[w_\lambda]_{\H}^2=[v_\lambda]_{\H}^2,\quad \|w_\lambda\|_{2_s^\ast}^{2_s^\ast}=\|v_\lambda\|_{2_s^\ast}^{2_s^\ast},\quad \|w_\lambda\|_2^2=\xi_\lambda^{-2s}\|v_\lambda\|_2^2,\quad \|w_\lambda\|_q^q=\xi_\lambda^{-N\frac{2_s^\ast-q}{2_s^\ast}}\|v_\lambda\|_q^q .$$
From Lemma~\ref{PS_theorem_2} we also obtain $|\log(R_\lambda^{(1)})|\to\infty$. Thus either $R_\lambda\to \infty$ or $R_\lambda\to 0$. If  $R_\lambda\to \infty$ then $\xi_\lambda\to\infty$ and therefore $\|w_\lambda\|_{L^q(\RR)}\to 0$ as $\lambda\to 0$ (since $\|v_\lambda\|_q$ is uniformly bounded). This contradicts the fact that $w_\lambda\to U_1$ in $L^q_{loc}$. Hence $R_\lambda\to 0$, which in turn implies $\xi_\lambda\to 0$. 
Further, by Corollary~\ref{m_0m_lambda1} we have
\begin{equation*}
\lambda^\sigma\gtrsim  [U_1]^2_{H^s(\RR)}-[v_\lambda]^2_{H^s(\RR)}\gtrsim \begin{cases}
        \lambda^\sigma\left(\ln\frac1{\lambda}\right)^{-\frac{4-q}{q-2}}, \quad\mbox{ if }\, N=4s\\
        \lambda^{\frac{2(N-2s)}{(N-2s)q-4s}},\quad\mbox{ if }\,  2s<N<4s,\, q>\frac{4s}{N-2s}.
    \end{cases}
    \end{equation*}  
    
     
\section{Asymptotic behaviour when $2<p<2_s^\ast$}\label{subcritical}
In this section we shall assume that $2<p<2_s^\ast.$ Recall that $u_\lambda$ is a positive ground state solution of the PDE 
$$\fra u+u=|u|^{p-2}u+\lambda |u|^{q-2}u,$$
with $2<q<p<2^*_s$ and $\max_{x}u_\lambda(x)=u_\lambda(0)$ (w.l.g). 
Putting $\lambda=0$, we have 
\begin{equation}\label{limit_equation_3}
    \fra u+u=|u|^{p-2}u.
\end{equation}
By \cite{FLS}, we know that \eqref{limit_equation_3} has a unique (up to translation) positive, radially symmetric, strictly decreasing ground state solution $u_0\in\H$. We wish to show that as $\lambda\to0$ the family $u_\lambda$ converges to the $u_0$ in $\H$. 

Recall that 
\begin{equation*}
        I_{\lambda}(u)=\frac12[u]_{\H}^2+\frac12\int_{\RR}|u|^2\,dx-\frac1{p}\int_{\RR}|u|^{p}\,dx-\frac{\lambda}{q}\int_{\RR}|u|^q\,dx.
\end{equation*}
is the energy functional associated to \eqref{PDE} and the energy functional associated to \eqref{limit_equation_3} is given by
\begin{equation}
    I_0(u):=\frac12[u]_{\H}^2+\frac12\|u\|_2^2-\frac1{p}\|u\|_{p}^{p}.
\end{equation}
We also have the Nehari manifolds associated to $I_\lambda$ and $I_0$ by 
\begin{align}
    \SS_{\lambda}&:=\{u\in\H\setminus\{0\}: [u]_{\H}^2+\|u\|_2^2=\|u\|_p^p+\lambda\|u\|_q^q\}\label{S_lambda}.\\
    \SS_0&:=\{u\in\H\setminus\{0\}: [u]_{\H}^2+\|u\|_2^2=\|u\|_p^p\}\label{S_0}.
\end{align}
And define, $$m_\lambda:=\inf_{u\in\SS_\lambda}I_\lambda(u),\quad m_0:=\inf_{u\in\SS_0}I_0(u).$$
As done in Lemma \ref{N_lambda_char} we can easily show that for $\lambda\geq0$, $$m_\lambda=\inf_{v\in\H\setminus\{0\}}\sup_{t\geq0}I_\lambda(tu)=\inf_{v\in\H\setminus\{0\}}\sup_{t\geq0}I_\lambda(u_t).$$ Moreover, we have $m_\lambda=I_\lambda(u_\lambda)=\sup_{t\geq0}I_\lambda(tu_\lambda)=\sup_{t\geq0}I_\lambda((u_\lambda)_t).$ For any $u\in\H$ with $\frac1p\|u\|_p^p+\frac1q\lambda\|u\|_q^q-\frac12\|u\|_2^2>0$, there exists a unique $t(u)>0$ such that $\frac{d}{dt}I_\lambda(u_t)\mid_{t(u)}=0$ and $\sup_{t\geq0}I_\lambda(u_t)=I_\lambda(u_{t(u)})$.
\begin{lemma}\label{subcrit_bound}
    The family $(u_\lambda)_{\lambda>0}$ is uniformly bounded in $\H$. For $\lambda$ small, $m_0-m_\lambda\sim \lambda.$
\end{lemma}
\begin{proof}
    It's easy to show that $m_0\geq m_\lambda$.
    \begin{equation*}
            m_0\geq m_\lambda=I_{\lambda}(u_\lambda)-\frac1q\langle I_{\lambda}'(u_\lambda),u_\lambda\rangle\geq \left(\frac12-\frac1q\right)\|u\|_{\H}^2
    \end{equation*}
    Thus, $(u_\lambda)$ is uniformly bounded in $\H$. Observe that, by Sobolev inequality, for some $C>0$ independent of $\lambda$, $$\frac{\|u_\lambda\|_q^q}{[u_\lambda]_{\H}^2}\leq C.$$
    By Pohozaev identity, $$\frac1p\|u_\lambda\|_p^p-\frac12\|u_\lambda\|_2^2=\frac12[u_\lambda]_{\H}^2-\frac{\lambda}q\|u\|_q^q\geq \frac12[u_\lambda]_{\H}^2(1-C\lambda).$$ For $\lambda$ sufficiently small, $\frac1p\|u_\lambda\|_p^p-\frac12\|u_\lambda\|_2^2>0$. Thus there exists a unique $t_\lambda>0$ such that $\sup_{t\geq0}I_0((u_\lambda)_t)=I_0((u_\lambda)_{t_\lambda})$. Further, $$t_\lambda^{2s}= \frac{\frac12[u_\lambda]_{\H}^2}{\frac1p\|u_\lambda\|_p^p-\frac12\|u_\lambda\|_2^2}\leq\frac1{1-C\lambda}.$$ Now,
\begin{equation*}
        \begin{aligned}
            m_0&\leq \sup_{t\geq0}I_0((u_\lambda)_t)=I_0((u_\lambda)_{t_\lambda})=I_\lambda((u_\lambda)_{t_\lambda})+\lambda t_\lambda^N\frac1q\|u_\lambda\|_q^q \leq m_\lambda+\lambda t_\lambda^N\frac1q\|u_\lambda\|_q^q.
        \end{aligned}
    \end{equation*}
Note that $t_\lambda^N\leq 1+C\lambda$ for some $C>0.$ Since $u_\lambda$ is bounded in $L^q(\RR)$, for some $C>0$, 
$$m_0-m_\lambda\leq C\lambda.$$
For the other direction, let $\tau_\lambda$ be the unique $t>0$ such that $\sup_{t\geq0}I_\lambda((u_0)_t)=I_\lambda((u_0)_{\tau_\lambda})$. Further, $\tau_\lambda^{2s}= \frac{\frac12[u_0]_{\H}^2}{\frac1p\|u_0\|_p^p+\frac{\lambda}q\|u_0\|_q^q-\frac12\|u_0\|_2^2}.$ 
By Pohozaev identity, $$\frac12[u_0]_{\H}^2=\frac1p\|u_0\|_p^p-\frac12\|u_0\|_2^2.$$
Thus,  $$\tau_\lambda^{2s}= \frac{\frac12[u_0]_{\H}^2}{\frac12[u_0]_{\H}^2+\frac{\lambda}q\|u_0\|_q^q}\geq\frac1{1+C\lambda},$$
where $C>0.$ Then  
\begin{equation*}
        \begin{aligned}
            m_\lambda&\leq \sup_{t\geq0}I_\lambda((u_0)_t)=I_\lambda((u_0)_{\tau_\lambda})=I_0((u_0)_{\tau_\lambda})-\lambda \tau_\lambda^N\frac1q\|u_0\|_q^q \leq m_0-\lambda \tau_\lambda^N\frac1q\|u_0\|_q^q.
        \end{aligned}
    \end{equation*}
Thus for small $\lambda>0$ and some $C_1>0$, 
$$m_0-m_\lambda\geq \lambda \tau_\lambda^N\frac1q\|u_0\|_q^q\geq C_1\lambda.$$
This concludes the lemma.
\end{proof}
\begin{lemma}\label{subcrit_asymp}For $\lambda$ small,
\begin{align}
    \|u_0\|_p^p-\|u_\lambda\|_p^p\sim\lambda.\\
    \|u_0\|_2^2-\|u_\lambda\|_2^2\sim\lambda.\label{5a}\\
    [u_0]_{\H}^2-[u_\lambda]_{\H}^2\sim\lambda.\label{5b}
\end{align}
\end{lemma}
\begin{proof}
By $u_\lambda\in\SS_\lambda$ and Pohozaev identity, 
    \begin{equation*}
        \begin{aligned}
            [u_\lambda]_{\H}^2+\|u_\lambda\|_2^2&=\|u_\lambda\|_p^p+\lambda\|u_\lambda\|_q^q.\\
            \frac1{2_s^\ast}[u_\lambda]_{\H}^2+\frac12\|u_\lambda\|_2^2&=\frac1p\|u_\lambda\|_p^p+\frac{\lambda}q\|u_\lambda\|_q^q.
        \end{aligned}
    \end{equation*}
\noindent Therefore we get
\begin{align*}
    [u_\lambda]_{\H}^2&=\frac{N(p-2)}{s2p}\|u_\lambda\|_p^p+\frac{N(q-2)}{s2q}\lambda\|u_\lambda\|_q^q, \ \|u_\lambda\|_2^2=\frac{N(2_s^\ast-p)}{s2_s^\ast p}\|u_\lambda\|_p^p+\frac{N(2_s^\ast-q)}{s2_s^\ast q}\lambda\|u_\lambda\|_q^q,
\end{align*}
and 
\begin{align*}
    m_\lambda =\frac1{2}[u_\lambda]_{\H}^2+\frac12\|u_\lambda\|_2^2-\frac1p\|u_\lambda\|_p^p-\frac{\lambda}q\|u_\lambda\|_q^q =\frac{p-2}{2p}\|u_\lambda\|_p^p+\lambda\frac{q-2}{2q}\|u_\lambda\|_q^q.
\end{align*}
\noindent Similarly, 
\begin{align*}
    [u_0]_{\H}^2=\frac{N(p-2)}{s2p}\|u_0\|_p^p, \ \|u_0\|_2^2 =\frac{N(2_s^\ast-p)}{s2_s^\ast p}\|u_0\|_p^p.
\end{align*}
and 
\begin{equation*}
    m_0=\frac{p-2}{2p}\|u_0\|_p^p.
\end{equation*}
 By the boundedness of $\{u_\lambda\}$ in $L^q(\RR)$,
$$m_0-m_\lambda=\frac{p-2}{2p}(\|u_0\|_p^p-\|u_\lambda\|_p^p)-\lambda\frac{q-2}{2q}\|u_\lambda\|_q^q\geq\frac{p-2}{2p}(\|u_0\|_p^p-\|u_\lambda\|_p^p)-C\lambda.$$ 
Thus, by Lemma \ref{subcrit_bound}, $\|u_0\|_p^p-\|u_\lambda\|_p^p\leq C\lambda$.For the other way inequality, we see that \begin{align*}
     \lambda \tau_\lambda^N\frac1q\|u_0\|_q^q&\leq m_0-m_\lambda=\frac{p-2}{2p}(\|u_0\|_p^p-\|u_\lambda\|_p^p)-\lambda\frac{q-2}{2q}\|u_\lambda\|_q^q\\ \Rightarrow\lambda \tau_\lambda^N\frac1q\|u_0\|_q^q&\leq \frac{p-2}{2p}(\|u_0\|_p^p-\|u_\lambda\|_p^p).
\end{align*} 
Thus, we get $\|u_0\|_p^p-\|u_\lambda\|_p^p\sim\lambda$.

Using the above relations and the Lemma \ref{subcrit_bound}, we can similarly prove \eqref{5a} and \eqref{5b}.
\end{proof}
{\bf Proof of Theorem~\ref{Asymp_2}}\begin{proof}
From Lemma~\ref{subcrit_asymp}, we have $\|u_\lambda\|_{\H}\to\|u_0\|_{\H}$ as $\lambda\to0$.  Since $\H$ is uniformly convex, using $u_\lambda\rightharpoonup u_0$ in $\H$ we get $u_\lambda\to u_0$ in $\H$. Finally, imitating the step 2 of the proof of Theorem \ref{Asymp_1}, we see that $u_\lambda\to u_0$ in $\H\cap C^{2s-\delta}(\RR)$ where $\delta\in(0,2s)$. 
\end{proof}

\section{Local uniqueness and nondegeneracy}\label{loc_uniq}

In this section we shall prove the local uniqueness theorem and nondegeneracy results, namely Theorem~\ref{t:uni-cri} and Theorem~\ref{t:uni-sub}.

\subsection{Proof of Theorem~\ref{t:uni-cri}}

We first prove the uniqueness part. To this end, suppose, there exist two different families of positive ground state solutions $u_\lambda^1$ and $u_{\lambda}^2$ of \eqref{PDE}. Now define,
$$v_\lambda^i(x)=\lambda^{\frac1{q-2}}u_\lambda^i\left(\lambda^{\frac{2_s^\ast-2}{2s(q-2)}}x\right),\quad i=1,2.$$
Therefore, by Theorem~\ref{Asymp_1}
$$\|v_\lambda^i-U_{\rho_0}\|_{\dot{H}^s(\RR)\cap C^{2s-\sigma}(\RR)}\to 0, \quad\text{as}\quad n\to\infty, \quad i=1, 2,$$ for some $\sigma\in(0,2s)$. Recall that $v_\lambda^i$, $i=1,2$ satisfies \eqref{Rescaled_PDE} 
$$\fra v_\lambda^i+\lambda^\sigma v_\lambda^i={\left(v_\lambda^i\right)}^{2_s^\ast-1}+\lambda^\sigma {\left(v_\lambda^i\right)}^{q-1}\quad\text{in }\RR.$$ Define
$$\theta_\lambda:=\frac{v_\lambda^1-v_\lambda^2}{\|v_\lambda^1-v_\lambda^2\|_{\infty}}.$$
Therefore $\theta_\lambda$ satisfies
\begin{equation}\label{6a}
    \fra \theta_\lambda+\lambda^\sigma \theta_\lambda=\left(c_\lambda^1(x)+\lambda^\sigma c_\lambda^2(x)\right)\theta_\lambda \text{ in } \RR,
\end{equation}
where 
\begin{align*}
    c_\lambda^1(x)&=(2_s^\ast-1)\int_0^1(tv_\lambda^1(x)+(1-t)v_\lambda^2(x))^{2_s^\ast-2}\,dt,\\
    c_\lambda^2(x)&=(q-1)\int_0^1(tv_\lambda^1(x)+(1-t)v_\lambda^2(x))^{q-2}\,dt.
\end{align*}
Since $v_\lambda^i\to U_{\rho_0}$ in $C^{2s-\sigma}(\RR)$, $i=1,2$ for some $\sigma\in(0,2s).$ Thus,
\begin{equation}\label{6c}
    c_\lambda^1(x)\to (2_s^\ast-1)U_{\rho_0}^{2_s^\ast-2}(x),\quad \lambda^\sigma c_\lambda^2(x)\to0\text{ as }\lambda\to0.
\end{equation}
Combining this with $\|\theta_\lambda\|_{\infty}=1$, we get that as $\lambda\to0$, $$|\fra\theta_\lambda(x)|\leq |c_\lambda^1(x)+\lambda^\sigma c_\lambda^2(x)|\,\|\theta_\lambda\|_{\infty}=(2_s^\ast-1)U_{\rho_0}^{2_s^\ast-2}(x)+o(1).$$
Hence $\fra\theta_\lambda$ is uniformly bounded for all $\lambda$ sufficiently small. By using Schauder estimate \cite[Theorem~1.1]{RS2}, and the fact that $C^{k,\alpha}_{\text{loc}}(\RR)$ is compactly embedded in $C^{k,\beta}_{\text{loc}}(\RR)$ for any $k\geq0$ and $0<\beta<\alpha\leq1$, we get $\theta_\lambda\to \theta$ in $C^{2s-\delta}_{\text{loc}}(\RR)$ for some $\delta\in(0,2s)$.\\
\noindent\textbf{Claim-1:} The family $\{\theta_\lambda\}_{\lambda\leq1}$ is uniformly bounded in $\HHH$.\\
Let $\beta<\frac{2_s^\ast(q-4)+4}{2}$ be a positive number. By Sobolev inequality and \eqref{6a}, we get 
\begin{equation}\label{6b}
\begin{aligned}
    S_\ast\left(\int_{\RR}\theta_\lambda^{2_s^\ast}\,dx\right)^{\frac{2}{2_s^\ast}}\leq [\theta_\lambda]_{\H}^2&=\int_{\RR}(c_\lambda^1(x)+\lambda^\sigma c_\lambda^2(x)-\lambda^\sigma)\theta_\lambda^2\,dx\\
    &\leq \int_{\RR}(c_\lambda^1(x)+\lambda^\sigma c_\lambda^2(x))\theta_\lambda^2\,dx\\
    &\leq \int_{\RR}(c_\lambda^1(x)+\lambda^\sigma c_\lambda^2(x))\theta_\lambda^{2-\beta}\,dx\\&\leq \left(\int_{\RR}\theta_\lambda^{2_s^\ast}\,dx\right)^{\frac{2-\beta}{2_s^\ast}}\left(\int_{\RR}(c_\lambda^1(x)+\lambda^\sigma c_\lambda^2(x))^{\frac{2_s^\ast}{2_s^\ast-2+\beta}}\,dx\right)^{\frac{2_s^\ast-2+\beta}{2_s^\ast}}.
\end{aligned}
\end{equation}
By Triangle inequality, $$S_\ast\|\theta_\lambda\|_{2_s^\ast}^{\beta}\leq \|c_\lambda^1+\lambda^\sigma c_\lambda^2\|_{{\frac{2_s^\ast}{2_s^\ast-2+\beta}}}\leq \|c_\lambda^1\|_{\frac{2_s^\ast}{2_s^\ast-2+\beta}}+\lambda^\sigma \|c_\lambda^2\|_{\frac{2_s^\ast}{2_s^\ast-2+\beta}}.$$
By Minkowski's integral inequality, 
$$\left(\int_{\RR}(c_\lambda^1(x))^{\frac{2_s^\ast}{2_s^\ast-2+\beta}}\,dx\right)^{\frac{2_s^\ast-2+\beta}{2_s^\ast}}\leq (2_s^\ast-1) \int_0^1\left(\int_{\RR}(tv_\lambda^1(x)+(1-t) v_\lambda^2(x))^{\frac{2_s^\ast(2_s^\ast-2)}{2_s^\ast-2+\beta}}\,dx\right)^{\frac{2_s^\ast-2+\beta}{\beta}}\,dt.$$
and 
$$\left(\int_{\RR}(c_\lambda^2(x))^{\frac{2_s^\ast}{2_s^\ast-2+\beta}}\,dx\right)^{\frac{2_s^\ast-2+\beta}{2_s^\ast}}\leq (q-1) \int_0^1\left(\int_{\RR}(tv_\lambda^1(x)+(1-t) v_\lambda^2(x))^{\frac{2_s^\ast(q-2)}{2_s^\ast-2+\beta}}\,dx\right)^{\frac{2_s^\ast-2+\beta}{\beta}}\,dt.$$
Since  $\beta<\frac{2_s^\ast(q-4)+4}{2}$  $\frac{2_s^\ast(2_s^\ast-2)}{2_s^\ast-2+\beta}>\frac{2_s^\ast(q-2)}{2_s^\ast-2+\beta}>2$. Recall that the families $\{v_\lambda^i\}_{\lambda\leq1}$, $i=1,2$ are uniformly bounded in $L^t(\RR)$ for all $t\geq2$. Thus  
the last term of \eqref{6b} is uniformly bounded which in turn implies that the family $\{\theta_\lambda\}_{\lambda\leq1}$ is uniformly bounded in $\HHH$. Hence $\theta_\lambda\rightharpoonup \theta$ in $\HHH$. Taking $\lambda\to0$ in the weak formulation of \eqref{6a}, we find that $\theta$ solves
\begin{equation}\label{6d}
    \fra \theta=(2_s^\ast-1)U_{\rho_0}^{2_s^\ast-2}\theta \text{ in }\RR.
\end{equation}
Since $\theta_\lambda\to \theta$ in $C^{0,\alpha}_{\text{loc}}(\RR)$ and $\|\theta_\lambda\|_{\infty}=1$ for all $\lambda$, we have $\|\theta\|_{\infty}\leq 1$.
It's easy to check that the rescale $\psi(x):=\theta(\rho_0x)$ solves 
\begin{equation}\label{6e}
    \fra \psi=(2_s^\ast-1)U_1^{2_s^\ast-2}\psi \text{ in }\RR.
\end{equation}
Now by \cite[Theorem~1.1]{DDS}, all bounded solutions of \eqref{6e} are linear combinations of the following $N+1$ functions:
$$\frac{N-2s}{2}U_1+x\cdot\nabla U_1,\quad  (U_1)_{x_i}, i=1,2,\ldots,N.$$
Thus there exist $c,c_i\in\RR$ such that 
$$\psi(x)=c\frac{1-|x|^2}{(1+|x|^2)^{\frac{N-2s+2}{2}}}+\sum_{i=1}^N c_i\frac{2x_i}{(1+|x|^2)^{\frac{N-2s+2}{2}}}.$$
Since $v_\lambda^i,i=1,2$ are symmetric so $\theta_\lambda,\theta$ and $\psi$ are symmetric. Thus each $c_i=0$.\\
\noindent\textbf{Claim-2:} $c=0$.\\
Suppose, $c\neq 0$. For simplicity, assume $c=1$. Thus, $$\psi(x)=\frac{1-|x|^2}{(1+|x|^2)^{\frac{N-2s+2}{2}}}.$$
By \eqref{norm_2,q},  
$$\|v_\lambda^i\|_2^2=\frac{2(2_s^\ast-q)}{q(2_s^\ast-2)}\|v_\lambda^i\|_q^q, \quad i=1,2.$$
Subtracting one equation from the other and multiplying $\frac{1}{\|v_\lambda^1-v_\lambda^2\|_{\infty}}$, we get 
$$\int_{\RR}\theta_\lambda (v_\lambda^1+v_\lambda^2)\,dx= \frac{2(2_s^\ast-q)}{(2_s^\ast-2)}\int_{\RR}\theta_\lambda\int_0^1(tv_\lambda^1+(1-t)v_\lambda^2)^{q-1}\,dt\,dx.$$
Since $\theta_\lambda\to\theta$ in $C^{2s-\delta}(\RR)$ and $v_\lambda^i\to U_{\rho_0}$ in $\H\cap C^{2s-\delta}(\RR)$, by DCT, $$\int_{\RR}\theta U_{\rho_0}\,dx=\frac{(2_s^\ast-q)}{(2_s^\ast-2)}\int_{\RR}\theta U_{\rho_0}^{q-1}\,dx.$$ 
By a simple change of variable, we get 
$$\int_{\RR}\psi(x)U_1(x)\,dx=\frac{(2_s^\ast-q)}{(2_s^\ast-2)}\rho_0^{-\frac{(N-2s)(q-2)}{2}}\int_{\RR}\psi(x)U_1^{q-1}(x)\,dx.$$
Using the expression of $\rho_0$, the above equation yields,
$$\int_{\RR}\psi(x)U_1(x)\,dx=\frac{q}{2}\frac{\|U_1\|_2^2}{\|U_1\|_q^q}\int_{\RR}\psi(x)U_1^{q-1}(x)\,dx.$$
By Lemma \ref{App_lemma_1}, 
$$\frac1{c_{N,s}}\left(\frac{-2N+(N-2s)2}{(N-2s)2}\right)\|U_1\|_2^2=\frac{q}{2}\frac{\|U_1\|_2^2}{\|U_1\|_q^q}\frac1{c_{N,s}}\left(\frac{-2N+(N-2s)q}{(N-2s)q}\right)\|U_1\|_q^q.$$
which implies that $q=2$ contradicting the assumption on $q$.
Hence, $\theta=0$ and $\theta_\lambda\to0$ on every compact set in $\RR$. Let $y_\lambda\in\RR$ be such that 
\begin{equation}\label{6f}
\theta_\lambda(y_\lambda)=\|\theta_\lambda\|_{\infty}=1.
\end{equation}
This implies that $y_\lambda\to\infty$ as $\lambda\to0$.\\
\noindent\textbf{Claim-3:} For $\lambda<1$ and $R>0$ large enough,
$$|\theta_\lambda(x)|\leq\frac{C}{|x|^{N-2s}}, \quad |x|>R,$$
where $C$ is independent of $\lambda$.

Since $v_\lambda\in L^r(\RR)$ for all $r\in[2,\infty)$, $\fra v_\lambda=v_\lambda^{2_s^\ast-1}+\lambda^\sigma (v_\lambda^{q-1}-v_\lambda)\in L^2(\RR)$. Which implies $$[v_\lambda]_{H^{2s}(\RR)}^2=\|\fra v_\lambda\|_2^2<\infty.$$
Hence $v_\lambda\in H^{2s}(\RR)$. By \cite[(C.8)]{FLS}, for any $f\in H^{2s}(\RR)$, $$\fra |f|\leq(\text{sgn}f)\fra f\quad\text{a.e. on }\RR,$$
where $\text{sgn } f(x)=\frac{f(x)}{|f(x)|}$ if $f(x)\neq0$ and $0$ if $f(x)=0$. Since $v_\lambda\in H^{2s}(\RR)$, $\theta_\lambda\in H^{2s}(\RR)$. Thus,
$$\fra |\theta_\lambda|\leq (\text{sgn }\theta_\lambda)\fra\theta_\lambda\leq \frac{\theta_\lambda}{|\theta_\lambda|}(c_\lambda^1+\lambda^\sigma c_\lambda^2)\theta_\lambda=(c_\lambda^1+\lambda^\sigma c_\lambda^2)|\theta_\lambda|\quad\text{a.e. on }\RR.$$

It's also easy to check that $|\theta_\lambda|\in\H$. We define the Kelvin transform of $|\theta_\lambda|$ by $$\Tilde{\theta}_\lambda(x):=\frac1{|x|^{N-2s}}|\theta_\lambda|\left(\frac{x}{|x|^2}\right),  \ x\neq0.$$
It can shown that $\Tilde{\theta}_\lambda\in\H$ and by \cite[Proposition~A.1]{RS1} $\fra\Tilde{\theta}_\lambda(x)=\frac1{|x|^{N+2s}}\fra|\theta_\lambda|(x)$. Therefore, we can see that for any $\lambda\leq1$, $\theta_\lambda$ is a subsolution of $$\fra u=\frac1{|x|^{4s}}\left(c_\lambda^1\left(\frac{x}{|x|^2}\right)+c_\lambda^2\left(\frac{x}{|x|^2}\right)\right)u.$$
Since $v_\lambda^{i},i=1,2$ is uniformly bounded in $\H$, the $L^2$ radial lemma yields, for any $x\neq0$, $$v_\lambda^i(x)\leq C(N)|x|^{-\frac N2}\|v_\lambda^i\|_{2}\leq C(N,s)|x|^{-\frac N2}.$$
For $r<1$ we have,
\begin{align*}
\int_{B_r}\frac1{|x|^{4s\gamma}}\left(c_\lambda^1\left(\frac{x}{|x|^2}\right)\right)^\gamma dx&=C\int_{B_r}\int_0^1\frac1{|x|^{4s\gamma}}\left(tv_\lambda^1\left(\frac{x}{|x|^2}\right)+(1-t)v_\lambda^2\left(\frac{x}{|x|^2}\right)\right)^{(2_s^\ast-2)\gamma}dt\,dx.\\
&\leq C\int_{B_r}|x|^{\frac{N(2_s^\ast-2)\gamma}2-4s\gamma}\,dx.
\end{align*}
If we choose a $\gamma\in(\frac{N}{2s},\frac{N}{4s-(2_s^\ast-2)\frac N2})$ then the last integral is finite.

For $r<1$ we also have,
\begin{align*}
\int_{B_r}\frac1{|x|^{4s\gamma}}\left(c_\lambda^2\left(\frac{x}{|x|^2}\right)\right)^\gamma dx&=C\int_{B_r}\int_0^1\frac1{|x|^{4s\gamma}}\left(tv_\lambda^1\left(\frac{x}{|x|^2}\right)+(1-t)v_\lambda^2\left(\frac{x}{|x|^2}\right)\right)^{(q-2)\gamma}dt\,dx.\\
&\leq C\int_{B_r}|x|^{\frac{N(q-2)\gamma}2-4s\gamma}\,dx.
\end{align*}
Since, $q>2+\frac{4s}N$, $\frac{N}{2s}<\frac{N}{4s-(q-2)\frac N2}$. If we choose a $\gamma\in(\frac{N}{2s},\frac{N}{4s-(q-2)\frac N2})$ then the last integral is finite.

Hence we can find a $\gamma>\frac N{2s}$ such that $\frac1{|x|^{4s}}\left(c_\lambda^1\left(\frac{x}{|x|^2}\right)+c_\lambda^2\left(\frac{x}{|x|^2}\right)\right) \in L^\gamma(B_r)$ for some $r<1.$ Now applying Moser iteration method in the spirit of \cite[Proposition~2.4]{JLX}, we get $\sup_{B_{r/2}}|\Tilde{\theta}_\lambda|\leq C\left(\int_{B_r}|\Tilde{\theta}_\lambda|^{2_s^\ast}\,dx\right)^{1/2_s^\ast}.$ Using Claim 1 we also have,
$$\int_{B_r}|\Tilde{\theta}_\lambda|^{2_s^\ast}\,dx\leq \int_{\RR}|\Tilde{\theta}_\lambda|^{2_s^\ast}\,dx\leq \int_{\RR}|\theta_\lambda|^{2_s^\ast}\,dx\leq C.$$

This implies that for some large $R>0$, $$|\theta_\lambda(x)|\leq \frac{C}{|x|^{N-2s}}, \quad |x|>R.$$

Therefore it follows that $\theta_\lambda(y_\lambda)\to0$ contradicting \eqref{6f}. Hence the uniqueness follows.

\noindent\textbf{Proof of nondegeneracy:} From the uniqueness part, there exists $0<\lambda_0\leq1$ such that if $\lambda<\lambda_0$ then \eqref{PDE} admits a unique radial positive ground state solution $u_\lambda$. We want to find a $\Tilde{\lambda}_0\leq \lambda_0$ such that for every $\lambda<\Tilde{\lambda}_0$, the equation $$\fra \phi+\phi=(2_s^\ast-1)u_\lambda^{2_s^\ast-2}\phi+\lambda(q-1)u_\lambda^{q-2}\phi\,\,\mbox{ in }\RR$$ has only trivial solution in $H^s_r(\RR)$. The proof is very similar to the uniqueness part. We will prove it by the method of contradiction. 

Suppose,  there exist $\{\lambda_n\}\subset(0,\lambda_0)$ 
and $\{\phi_n\}$ such that 
$\lambda_n\
\to 0,\, \phi_n\in 
H^s_r(\RR)\setminus\{0\}$ 
and $$\fra \phi_n+\phi_n=
(2_s^\ast-
1)u_{\lambda_n}^{2_s^\ast-
2}\phi_n+\lambda_n(q-
1)u_{\lambda_n}^{q-2}\phi_n\, \, \mbox{ in }\, \RR.$$

Define the rescaling of $\phi_n$ by $$\psi_n(x):=\lambda_n^{\frac1{q-2}}\phi_n\left(\lambda_n^{\frac{2_s^\ast-2}{2s(q-2)}}x\right).$$
It is easy to check that $\psi_n$ satisfies
\begin{equation*}
    \fra \psi_n+\lambda_n^\sigma\psi_n=
(2_s^\ast-
1)v_{\lambda_n}^{2_s^\ast-
2}\psi_n+\lambda_n^\sigma(q-
1)v_{\lambda_n}^{q-2}\psi_n\text{ in }\RR,
\end{equation*}
where $v_n(x):=\lambda_n^{\frac1{q-2}}u_n\left(\lambda_n^{\frac{2_s^\ast-2}{2s(q-2)}}x\right).$ Now define, $\tilde\psi_n=\frac{\psi_n}{\|\psi_n\|_{\infty}}$. Then $\tilde\psi_n$ satisfies
\begin{equation}\label{6g}
    \fra \tilde\psi_n+\lambda_n^\sigma\tilde\psi_n=
(2_s^\ast-
1)v_{\lambda_n}^{2_s^\ast-
2}\tilde\psi_n+\lambda_n^\sigma(q-
1)v_{\lambda_n}^{q-2}\tilde\psi_n\,\, \mbox{ in }\, \RR.
\end{equation}
From the proof of Theorem~\ref{Asymp_1} (step 2), we know that $v_\lambda$ is uniformly bounded (w.r.t. $\lambda\leq1$) in $L^\infty(\RR)$.
Therefore, $\fra \tilde\psi_n$ is uniformly bounded in $L^{\infty}(\RR)$. Thus, applying Schauder estimate \cite[Theorem 1.1]{RS2}, we have $$\tilde\psi_n\to\tilde\psi\text{ in }C_{\text{loc}}^{\alpha}(\RR)$$ for some $\alpha\in(0,1).$ As done in the Claim 1 in the uniqueness part, we can show that $\{\tilde\psi_n\}$ is uniformly bounded in $\HHH$. Hence $\tilde\psi_n\rightharpoonup \tilde\psi$ in $\HHH$. Taking $\lambda_n\to0$ in the weak formulation of \eqref{6g}, we find that $\tilde\psi$ solves
\begin{equation}
    \fra \tilde\psi=(2_s^\ast-1)U_{\rho_0}^{2_s^\ast-2}\tilde\psi \text{ in }\RR,
\end{equation}
and $\|\tilde\psi\|_{\infty}\leq1$. Since the sequence $\phi_n$ is radially symmetric, $\tilde\psi_n$ and $\tilde\psi$ are also symmetric. Thus using \cite[Theorem~1.1]{DDS}, as in the uniqueness part, we have $$\tilde\psi(x)=c\rho_0^{N-2s}\frac{\rho_0^2-|x|^2}{(\rho_0^2+|x|^2)^{\frac{N-2s+2}{2}}}$$ for some $c\in\R.$\\
\noindent\textbf{Claim:} $c=0$.\\
Set $w_n(x)=x\cdot\nabla v_{\lambda_n}(x)$. Then by Lemma \ref{App_lemma_2},
\begin{align*}
    \fra w_n(x)&=x\cdot\nabla(\fra v_{\lambda_n})(x)+2s\fra v_{\lambda_n}(x)\\
    &=\left[(2_s^\ast-1)v_{\lambda_n}^{2_s^\ast-2}w_n+\lambda_n^\sigma(q-1)v_{\lambda_n}^{q-2}w_n-\lambda^\sigma w_n\right]+2s\left[v_{\lambda_n}^{2_s^\ast-1}+\lambda_n^\sigma v_{\lambda_n}^{q-1}-\lambda_n^\sigma v_{\lambda_n}\right].
\end{align*}
Thus $w_n$ satisfies 
\begin{equation}\label{6h}
    \fra w_n+\lambda^\sigma w_n=\left[(2_s^\ast-1)v_{\lambda_n}^{2_s^\ast-2}w_n+\lambda_n^\sigma(q-1)v_{\lambda_n}^{q-2}w_n\right]+2s\left[v_{\lambda_n}^{2_s^\ast-1}+\lambda_n^\sigma v_{\lambda_n}^{q-1}-\lambda_n^\sigma v_{\lambda_n}\right].
\end{equation}
Taking $w_n$ as the test function for \eqref{6g} and $\tilde\psi_n$ for \eqref{6h}  yields
\begin{equation}\label{6i}
    \int_{\RR}\left[v_{\lambda_n}^{2_s^\ast-1}+\lambda_n^\sigma v_{\lambda_n}^{q-1}-\lambda_n^\sigma v_{\lambda_n}\right]\tilde\psi_n\,dx=0.
\end{equation}
Recall that $v_{\lambda_n}$ satisfies 
$$\fra v_{\lambda_n}+\lambda_n^\sigma v_{\lambda_n}=v_{\lambda_n}^{2_s^\ast-1}+\lambda_n^\sigma v_{\lambda_n}^{q-1}\, \mbox{ in }\, \RR.$$
Multiplying the above equation by $\tilde\psi_n$ and \eqref{6g} by $v_{\lambda_n}$ and then integrating by parts, we obtain
$$(2_s^\ast-1)\int_{\RR}v_{\lambda_n}^{2_s^\ast-1}\tilde\psi_n\,dx+\lambda_n^\sigma(q-1)\int_{\RR}v_{\lambda_n}^{q-1}\tilde\psi_n\,dx=\int_{\RR}v_{\lambda_n}^{2_s^\ast-1}\tilde\psi_n\,dx+\lambda_n^\sigma\int_{\RR}v_{\lambda_n}^{q-1}\tilde\psi_n\,dx.$$
Thus, $$\int_{\RR}v_{\lambda_n}^{2_s^\ast-1}\tilde\psi_n\,dx=-\lambda_n^\sigma\frac{q-2}{2_s^\ast-2}\int_{\RR}v_{\lambda_n}^{q-1}\tilde\psi_n\,dx.$$
Combining this with \eqref{6i}, we get 
$$\int_{\RR}v_{\lambda_n}\tilde\psi_n\,dx=\frac{2_s^\ast-q}{2_s^\ast-2}\int_{\RR}v_{\lambda_n}^{q-1}\tilde\psi_n\,dx.$$
Taking  the limit $n\to\infty$ to the above equation yields $$\int_{\RR}U_{\rho_0}\tilde\psi\,dx=\frac{2_s^\ast-q}{2_s^\ast-2}\int_{\RR}U_{\rho_0}^{q-1}\tilde\psi\,dx.$$
As done in the Claim 2 of the uniqueness part, the above equation is equivalent to $$\int_{\RR}\frac{1-|x|^2}{(1+|x|^2)^{\frac{N-2s+2}{2}}}U_1(x)\,dx=\frac{q}{2}\frac{\|U_1\|_2^2}{\|U_1\|_q^q}\int_{\RR}\frac{1-|x|^2}{(1+|x|^2)^{\frac{N-2s+2}{2}}}U_1^{q-1}(x)\,dx,$$
which in turn implies $q=2$ (by Lemma \ref{App_lemma_1}) and hence we get a contradiction as $q\neq 2$. Thus $\tilde\psi=0$ and $\tilde\psi_n\to0$ on every compact subset of $\RR$. Let $y_n$ be such that \begin{equation}\label{6j}
   \tilde\psi_n(y_n)=\|\tilde\psi_n\|_{\infty}=1.
\end{equation}
This implies that $y_n\to\infty$ as $n\to\infty$. Repeating the same argument as we did for the Claim~$3$ of the uniqueness part, it holds
 $$|\tilde\psi_n(x)|\leq \frac{C}{|x|^{N+2s}}, \quad |x|>R$$
 for some large $R>0$, where $C$ is independent of $\lambda$. Hence $\tilde\psi_n(y_n)\to 0$ contradicting \eqref{6j}. \hfill$\qed$

\subsection{Proof of Theorem~\ref{t:uni-sub}}
Let $2<q<p<2^*_s$ and set 
$$\mathcal{X}:=L^2_{rad}(\RR)\cap L^p_{rad}(\RR)$$
equipped with the norm $\|u\|_\mathcal{X}:=\|u\|_2+\|u\|_p$. If $u\in \mathcal{X}$ solves \eqref{PDE} in the sense of distribution, then by bootstrap method it follows that $u\in H^{2s+1}(\RR)$ (see \cite[Section 8.1]{FLS}).

Now define the function $F: \mathcal{X}\times(0,\infty)\to \mathcal{X}$ by
$$F(v,\lambda):=v- \big((-\Delta)^s+1\big)^{-1}(|v|^{p-2}v+\lambda|v|^{q-2}v).$$ Thus $F(u,\lambda)=0$ if{}f $(u, \lambda)$ satisfies \eqref{PDE}. Further, $F(u_0,0)=0$ and $F\in C^1$ (see \cite[Section 8.1]{FLS}). Then
$$\frac{\partial F}{\partial v}(u_0,0)=1- \big((-\Delta)^s+1\big)^{-1}\bigg((p-1)|u_0|^{p-2}\bigg).$$
Denote $$K:=- \big((-\Delta)^s+1\big)^{-1}\bigg((p-1)|u_0|^{p-2}\bigg).$$ Since $H^s(\RR)\hookrightarrow L^2(\RR, u_0^{p-2})$ is compact (since $u_0\in L^\infty(\RR)$ and $u_0\to 0$ as $|x|\to\infty$), 
it follows that $K$ is a compact  operator. Moreover, the nondegenracy result \cite[Theorem 3.3]{FLS} implies $\text{Ker}\bigg({\frac{\partial F}{\partial v}(u_0,0)}\bigg|_{H^s_{rad}}\bigg)=\{0\}$. Thus  $-1\not\in\sigma(K)$ (spectrum of $K$). Hence, from the implicit function theorem, we observe that there exists $\lambda_0>0$ such that for  $\lambda\in (0,\lambda_0)$ admits a unique radial positive ground state to \eqref{PDE} which is nondegenerate in $H^s_{rad}(\RR)$.
\hfill$\square$

\appendix
\section{Useful identities}
Recall that for any $x,y>0$, Beta function can be defined by $$B(x,y)=\int_0^{\infty}\frac{t^{x-1}}{(1+t)^{x+y}}\,dt.$$
and it has the property $$B(x,y+1)=B(x,y)\frac{y}{x+y},\quad B(x+1,y)=B(x,y)\frac{x}{x+y}.$$
By a simple change of variable we can see that $$\int_0^{\infty}\frac{x^m}{(1+x^2)^{n}}\,dx=\frac12\int_0^{\infty}\frac{t^{\frac{m-1}2}}{(1+t)^n}\,dt=\frac12B\left(\frac{m+1}{2},n-\frac{m+1}{2}\right).$$
\begin{lemma}\label{App_lemma_1}
    Let $\psi(x)=\frac{1-|x|^2}{(1+|x|^2)^{\frac{N-2s+2}{2}}}$. For any $t\in[2,2_s^\ast]$, 
\begin{equation}
    \int_{\RR}\psi U_1^{t-1}\,dx=\frac1{c_{N,s}}\left(\frac{-2N+(N-2s)t}{(N-2s)t}\right)\int_{\RR}U_1^t\,dx,
\end{equation}
where $c_{N,s}$ is the constant appearing in the definition \eqref{Talenti} of $U_1$.
\end{lemma} 
\begin{proof}
Using the Beta function we can see that
\begin{equation}\label{App_eq_1}
\int_{\RR}U_1^t\,dx=c_{N,s}^t\omega_N\int_0^{\infty}\frac{r^{N-1}}{(1+r^2)^{\frac{(N-2s)t}{2}}}\,dr=\frac{c_{N,s}^t\omega_N}{2}B\left(\frac N2,\frac{-N+(N-2s)t}{2}\right).
\end{equation}
Notice $\frac{-N+(N-2s)t}2>0\iff (N-2s)t>N$. Since $N>4s$ $\frac{-N+(N-2s)t}2>0$ for all $t\in[2,2_s^\ast]$. Similarly the LHS becomes
\begin{align}
    \int_{\RR}\psi U_1^{t-1}\,dx&=c_{N,s}^{t-1}\omega_N\int_0^{\infty}\frac{(1-r^2)r^{N-1}}{(1+r^2)^{1+\frac{(N-2s)t}{2}}}\,dr\nonumber\\
    &=c_{N,s}^{t-1}\omega_N\left[\int_0^{\infty}\frac{r^{N-1}}{(1+r^2)^{1+\frac{(N-2s)t}{2}}}\,dr-\int_0^{\infty}\frac{r^{N+1}}{(1+r^2)^{1+\frac{(N-2s)t}{2}}}\,dr\right]\nonumber\\
    &=\frac{c_{N,s}^{t-1}\omega_N}{2}\left[B\left(\frac N2,\frac{-N+(N-2s)t}{2}+1\right)-B\left(\frac N2+1,\frac{-N+(N-2s)t}{2}\right)\right]\nonumber\\
    &=\frac{c_{N,s}^{t-1}\omega_N}{2}\left(\frac{-2N+(N-2s)t}{(N-2s)t}\right)B\left(\frac N2,\frac{-N+(N-2s)t}{2}\right).\label{App_eq_2}
\end{align}
Using \eqref{App_eq_1} and \eqref{App_eq_2}, we get the lemma.
\end{proof}
 Let us define $u_t(x)=u(xt)$ and $w(x)=x\cdot\nabla u(x)$. It is easy to check that $w(x)=\lim_{t\to1^+}\frac{d}{dt}u_t(x)$.
\begin{lemma}\label{App_lemma_2}
    Suppose $u\in C^\infty(\RR)\cap L^{\infty}(\RR)\cap\H$. Then 
    \begin{equation}
        \fra w(x)=x\cdot\nabla(\fra u)(x)+2s\fra u(x), \quad x\in\RR.
    \end{equation}
\end{lemma}
\begin{proof}\phantom{\qedhere}
    \begin{align*}
        \fra w(x)&=\fra\left(\lim_{t\to1^+}\frac{d}{dt}u_t(x)\right)\\
        &=\lim_{t\to1^+}\frac{d}{dt}\fra u_t(x)\quad(\text{by DCT})\\
        &=\lim_{t\to1^+}\frac{d}{dt}\left(t^{2s}(\fra u)(xt)\right)\\
        &=\lim_{t\to1^+}\left[2st^{2s-1}\fra u(xt)+t^{2s}x\cdot\nabla(\fra u)(xt)\right]\\
        &= x\cdot\nabla(\fra u)(x)+2s\fra u(x).\tag*{\qed}
    \end{align*}
\end{proof}

\section*{Declaration} 
\noindent \textbf{Funding.}	 The research of M.~Bhakta is partially supported by DST Swarnajaynti fellowship (SB/SJF/2021-22/09). The research of P.~Das is partially supported by the NBHM grant 0203/5(38)/2024-R\&D-II/11224. The research of D.~Ganguly is partially supported by the SERB MATRICS (MTR/2023/000331).

\noindent \textbf{Competing interests.} The authors have no competing interests to declare that are relevant to the content of this article.\\
\noindent\textbf{Data availability statement.} Data sharing is not applicable to this article as no data sets were generated or analysed during the current study.
\bibliographystyle{abbrv}

\end{document}